\documentclass{conm-p-l}

\usepackage{amssymb}    
\usepackage{amsthm}   
\usepackage{enumerate}    
\usepackage{graphicx}   
\usepackage{mathrsfs}   

\usepackage{amsmath}
\usepackage{amsfonts}
\usepackage{amssymb}
\usepackage[mathcal]{euscript}
\usepackage{pb-diagram}
\usepackage[pdftex,pagebackref,letterpaper=true,colorlinks=true,pdfpagemode=none,
urlcolor=black,linkcolor=black,citecolor=black,pdfstartview=FitH]{hyperref}

\usepackage{tikz,bookmark,geometry}

\newcommand{\twiddle}{\raisebox{1pt}{\scalebox{.75}{$\mathord{\sim}$}}}

\theoremstyle{definition}
\newtheorem{definition}{Definition}[section]

\newtheorem{question}[definition]{Question}

\newtheorem{remark}[definition]{Remark}

\theoremstyle{plain}
\newtheorem{theorem}[definition]{Theorem}
\newtheorem*{theorem*}{Theorem}
\newtheorem{lemma}[definition]{Lemma}
\newtheorem{corollary}[definition]{Corollary}
\newtheorem{proposition}[definition]{Proposition}

\newtheorem*{claim*}{Claim}
\newtheorem{fact}[definition]{Fact}

\begin{document}

\begin{abstract}
We study the topological version of the partition calculus in the setting of countable ordinals. Let $\alpha$ and $\beta$ be ordinals and let $k$ be a positive integer. We write 
$\beta\to_{top}(\alpha,k)^2$ to mean that, for every red-blue coloring of the collection of 2-sized subsets of $\beta$, there is either a red-homogeneous set homeomorphic to $\alpha$ or 
a blue-homogeneous set of size $k$. The least such $\beta$ is the topological Ramsey number $R^{top}(\alpha,k)$.

We prove a topological version of the Erd\H{o}s-Milner theorem, namely that $R^{top}(\alpha,k)$ is countable whenever $\alpha$ is countable. More precisely, we prove that 
$R^{top}(\omega^{\omega^\beta},k+1)\leq\omega^{\omega^{\beta\cdot k}}$ for all countable ordinals $\beta$ and finite $k$. Our proof is modeled on a new easy proof of a weak version 
of the Erd\H{o}s-Milner theorem that may be of independent interest.

We also provide more careful upper bounds for certain small values of $\alpha$, proving among other results that $R^{top}(\omega+1,k+1)=\omega^k+1$, $R^{top}(\alpha,k)<
\omega^\omega$ whenever $\alpha<\omega^2$, $R^{top}(\omega^2,k)\leq\omega^\omega$ and $R^{top}(\omega^2+1,k+2)\leq\omega^{\omega\cdot k}+1$ for all finite $k$.

Our computations use a variety of techniques, including a topological pigeonhole principle for ordinals, considerations of a tree ordering based on the Cantor normal form of ordinals, 
and some ultrafilter arguments.
\end{abstract}

\author{Andr\'{e}s Eduardo Caicedo and Jacob Hilton}
\address{
  Andr\'es Eduardo Caicedo \\
  Boise State University \\
  Department of Mathematics \\
  1910 University Drive \\
  Boise, ID 83725-1555 \\
  USA
 }
\curraddr{Mathematical Reviews \\
  416 Fourth Street \\
  Ann Arbor, MI 48103-4820 \\
  USA
}
\email{aec@ams.org}
\urladdr{\href{http://www-personal.umich.edu/~caicedo/}{http://www-personal.umich.edu/{\twiddle}caicedo/}}

\address{
 Jacob Hilton \\
 School of Mathematics \\
 University of Leeds \\
 Leeds LS2 9JT \\
 UK}
\email{mmjhh@leeds.ac.uk}
\thanks{The second author's research was conducted under the supervision of John K. Truss with the support of an EPSRC Doctoral Training Grant Studentship.}



\date{\today}

\dedicatory{To W. Hugh Woodin, on the occasion of his birthday.}

\title[Topological Ramsey numbers]{Topological Ramsey numbers and countable ordinals}

\maketitle

\tableofcontents

\section{Introduction}

Motivated by Ramsey's theorem, the partition calculus for cardinals and Rado's arrow notation were introduced by Erd\H{o}s and Rado in \cite{MR0058687}. The version for ordinals, 
where the homogeneous set must have the correct order type, first appears in their seminal paper \cite{erdosrado}. In the following definition, $[X]^n$ denotes the set of subsets of $X$ 
of size $n$.

\begin{definition}
Let $\kappa$ be a cardinal, let $n$ be a positive integer, and let $\beta$ and all $\alpha_i$ be ordinals for $i\in\kappa$. We write
\[\beta\to(\alpha_i)^n_{i\in\kappa}\]
to mean that for every function $c:[\beta]^n\to\kappa$ (a \emph{coloring}) there exists some subset $X\subseteq\beta$ and some $i\in\kappa$ such that $X$ is \emph{an 
$i$-homogeneous copy of $\alpha_i$}, i.e., $[X]^n\subseteq c^{-1}(\{i\})$ and $X$ is order-isomorphic to $\alpha_i$.
\end{definition}

This relation has been extensively studied, particularly the case $n=2$, where we define the (\emph{classical}) \emph{ordinal Ramsey number} $R(\alpha_i)_{i\in\kappa}$ to be the least 
ordinal $\beta$ such that $\beta\to(\alpha_i)^2_{i\in\kappa}$. Ramsey's theorem can be stated using this notation as $R(\omega,\omega)=\omega$.

A simple argument that goes back to Sierpi\'nski shows that $\beta\not\to(\omega+1,\omega)^2$ for every countable ordinal $\beta$ \cite{sierpinski}, which means that if $\alpha >
\omega$ and $R(\alpha,\gamma)$ is countable, then $\gamma$ must be finite (see also \cite[Theorem 19]{erdosrado} and \cite[Theorem 4]{specker}). On the other hand, Erd\H{o}s and 
Milner showed that indeed $R(\alpha,k)$ is countable whenever $\alpha$ is countable and $k$ is finite \cite{erdosmilner}. Much work has been done to compute these countable ordinal 
Ramsey numbers. In particular, as announced without proof by Haddad and Sabbagh \cite{HS,haddadsabbagh,HS2}, there are algorithms for computing $R(\alpha,k)$ for several 
classes of ordinals $\alpha<\omega^\omega$ and all finite $k$; details are given in \cite{caicedohs} for the case $\alpha<\omega^2$ and in \cite{Mil} for the case $\alpha=\omega^m$ 
for finite $m$. See also \cite[Chapter 7]{W}, \cite{hajnallarson}, \cite{schipperusord} and \cite{weinert}.

With the development of structural Ramsey theory, many variants of the basic definition above have been introduced. In this article we are concerned with a ``topological'' version and a 
closely-related ``closed'' version. (Jean Larson suggested to call the second version ``limit closed''.) Both of these make use of the topological structure of an ordinal, which we endow 
with the order topology. Although the second version may appear more natural, it is the first that has been considered historically, since it can be defined for arbitrary topological spaces; 
the second version additionally requires an order structure.

In the following definition, $\cong$ denotes the homeomorphism relation. Following Baumgartner \cite{baumgartner}, we also say that a subspace $X$ of an ordinal is 
\emph{order-homeomorphic} to an ordinal $\alpha$ to mean that there is a bijection $X\to\alpha$ that is both an order-isomorphism and a homeomorphism, or equivalently, $X$ is both 
order-isomorphic to $\alpha$ and closed in its supremum.

\begin{definition}
Let $\kappa$ be a cardinal, let $n$ be a positive integer, and let $\beta$ and all $\alpha_i$ be ordinals for $i\in\kappa$.

We write
\[\beta\to_{top}(\alpha_i)^n_{i\in\kappa}\]
to mean that for every function $c:[\beta]^n\to\kappa$ there exists some subspace $X\subseteq\beta$ and some $i\in\kappa$ such that $X$ is an \emph{$i$-homogeneous topological 
copy of $\alpha_i$}, i.e., $[X]^n\subseteq c^{-1}(\{i\})$ and $X\cong\alpha_i$.

We write
\[\beta\to_{cl}(\alpha_i)^n_{i\in\kappa}\]
to mean that for every function $c:[\beta]^n\to\kappa$ there exists some subset $X\subseteq\beta$ and some $i\in\kappa$ such that $X$ is an \emph{$i$-homogeneous closed copy of 
$\alpha_i$}, i.e., $[X]^n\subseteq c^{-1}(\{i\})$ and $X$ is order-homeomorphic to $\alpha_i$.

Note that in both cases, $c$ is arbitrary (no continuity or definability is required).
\end{definition}

The two versions coincide in many cases, in particular for ordinals of the form $\omega^\gamma$ or $\omega^\gamma\cdot m+1$ with $m\in\omega$ \cite[Theorem 2.16]{ttppfo}. But in 
general they may differ; for example, $\omega+2$ is homeomorphic but not order-homeomorphic to $\omega+1$, and thus $\omega+1\to_{top}(\omega+2)^1_1$ while $\omega+1
\not\to_{cl}(\omega+2)^1_1$.

We focus primarily on the case $n=2$, where we define the \emph{topological ordinal Ramsey number} $R^{top}(\alpha_i)_{i\in\kappa}$ to be the least ordinal $\beta$ (when one 
exists) such that $\beta\to_{top}(\alpha_i)_{i\in\kappa}$, and the \emph{closed ordinal Ramsey number} $R^{cl}(\alpha_i)_{i\in\kappa}$ likewise. In particular, we study $R^{top}
(\alpha,k)$ and $R^{cl}(\alpha,k)$ when $\alpha$ is a countable ordinal and $k$ a positive integer, which have not previously been explored.

An important ingredient in our work is the case $n=1$, which was first studied by Baumgartner and Weiss \cite{baumgartner}. Analogously to Ramsey numbers, we define the 
\emph{pigeonhole number} $P(\alpha_i)_{i\in\kappa}$ to be the least ordinal $\beta$ such that $\beta\to(\alpha_i)^1_{i\in\kappa}$, and we define the \emph{topological pigeonhole 
number} $P^{top}(\alpha_i)_{i\in\kappa}$ and the \emph{closed pigeonhole number} $P^{cl}(\alpha_i)_{i\in\kappa}$ in a similar fashion. An algorithm for finding the topological 
pigeonhole numbers is given in \cite{ttppfo}, and we will describe the modifications required to obtain the closed pigeonhole numbers (see Section \ref{section:pcl}). Repeatedly we 
make use of the classical fact that the \emph{indecomposable ordinals}, the powers of $\omega$, satisfy $P(\omega^\alpha)_k=\omega^\alpha$ for all finite $k$. 

Returning to the case $n=2$, previous work has tackled uncountable Ramsey numbers. Erd\H{o}s and Rado \cite{erdosrado} showed that $\omega_1\to(\omega_1,\omega+1)^2$. 
Laver noted in \cite{laver} (and a proof, using a pressing down argument, is described for example in \cite{topologicalbh}) that one actually has $ \omega_1\to(\text{Stationary},
\mathrm{top}\,\omega+1)^2$, meaning that one can ensure either a $0$-homogeneous stationary subset or a $1$-homogeneous topological copy of $\omega+1$. Since every 
stationary subset of $\omega_1$ contains topological copies of every $\alpha\in\omega_1$ \cite{friedman}, this therefore shows that $\omega_1\to_{top}(\alpha,\omega+1)^2$ for all 
countable $\alpha$. In turn, this result was later extended by Schipperus using elementary submodel techniques to show that $\omega_1\to_{top}(\alpha)^2_k$ for all $\alpha\in
\omega_1$ and all finite $k$ \cite{topologicalbh} (the topological Baumgartner-Hajnal theorem). Meanwhile, both $\omega_1\to(\omega_1,\alpha)^2$ for all $\alpha\in\omega_1$ 
\cite{todorcevic} and $\omega_1\not\to(\omega_1,\omega+2)^2$ \cite{hajnal} are consistent with $\mathsf{ZFC}$, though the topological version of the former remains unchecked. 
Finally, it is also known that $\beta\not\to_{top}(\omega+1)^2_{\aleph_0}$ for all ordinals $\beta$ \cite{weiss}.

When $n>2$, not much more can be said in the setting of countable ordinals, since Kruse showed that if $n\geq 3$ and $\beta$ is a countable ordinal, then 
$\beta\not\to(\omega+1,n+1)^n$ \cite{MR0214476}. As for the ordinal $\omega_1$, \cite[Theorem 39 (ii)]{erdosrado} shows that $\omega_1\to(\omega+1)^n_k$ for all finite $n,k$ and, 
in fact, it seems to be a folklore result that $\omega_1\to_{top}(\omega+1)^n_k$ for all finite $n,k$;  a proof can be found in \cite{MR1058798}. On the other hand, we have the negative 
relations $\omega_1\not\to(n+1)^n_\omega$ for $n\ge 2$ \cite{erdosrado}, $\omega_1\not\to(\omega+2,n+1)^n$ for $n\ge 4$ \cite{MR0214476}, $\omega_1\not\to(\omega+2, 
\omega)^3$ \cite{MR1755613} and $\omega_1\not\to(\omega_1,4)^3$ \cite{MR0160734}. All that remains to be settled is the conjecture that $\omega_1\to(\alpha,k)^3$ for all $\alpha
\in\omega_1$ and all finite $k$, with the strongest result to date being that $\omega_1\to(\omega\cdot 2+1,k)^3$ for all finite $k$ \cite{albin}, building on the techniques of 
\cite{todorcevic}. The topological version of this result remains unexplored. In a different direction, Rosenberg \cite{rosenberg} has recently verified that $X\to_{top}(\omega+1)^n_k$ for 
all finite $n,k$ whenever $X$ is an uncountable separable metric space. Apparently, this last result has been discovered a few times. In particular, Stevo Todorcevic has informed us 
that in unpublished work from around 1990, he and Weiss characterized those metric spaces $X$ for which $X\to_{top}(\omega+1)^r_k$ for all finite $r,k$. He also provided us with a 
copy of his short unpublished note \cite{todorcevic96} from 1996, where he shows that a regular space $X$ with a point-countable basis is left-separated iff $X\not\to_{top}
(\omega+1)^2_2$ iff for some integers $r,k\ge2$, $X\not\to_{top}(\omega+1)^r_k$.

For a general introduction to the partition calculus of countable ordinals, see \cite{C}, which also provides some context for the work we present here. For a general introduction to 
topological Ramsey theory, see \cite{weiss}. In a sense, \cite{baumgartner} is the direct precursor of this paper and \cite{ttppfo}. Both authors were working independently on this topic 
when the talk on which \cite{C} is based was given. Upon learning of each other's research, we decided to combine and extend our notes. Among other extensions, this resulted in the 
weak topological Erd\H{o}s-Milner theorem.

\subsection{Our results}

Our main result is the weak topological Erd\H{o}s-Milner theorem:
\begin{theorem*}
Let $\alpha$ and $\beta$ be countable nonzero ordinals, and let $k>1$ be a positive integer. If
\[\omega^{\omega^\alpha}\to_{top}(\omega^\beta,k)^2,\]
then
\[\omega^{\omega^\alpha\cdot\beta}\to_{top}(\omega^\beta,k+1)^2.\]
\end{theorem*}
This is Theorem \ref{theorem:wtem}. Since trivially $\omega^{\omega^\alpha}\to_{top}(\omega^{\omega^\alpha},2)^2$, it follows by induction on $k$ that $R^{top}(\omega^
{\omega^\alpha},k+1)\leq\omega^{\omega^{\alpha\cdot k}}$. Hence the ordinals $R^{top}(\alpha,k)$ (and $R^{cl}(\alpha,k)$, since the two versions coincide when $\alpha$ is a power of 
$\omega$) are countable for all countable $\alpha$ and all finite $k$. (The reason why we use here the adjective \emph{weak} is discussed in Section \ref{section:wtem}.)

The paper is organized as follows.

Section \ref{section:preliminaries} completes the preliminaries, including further details of the relationship between the topological and closed partition relations.

In Section \ref{section:pcl} we introduce the case $n=1$ and solve the closed pigeonhole principle for ordinals: given a cardinal $\kappa$ and an ordinal $\alpha_i\geq 2$ for each 
$i\in\kappa$, we provide an algorithm to compute $P^{cl}(\alpha_i)_{i\in\kappa}$. This is presented in Theorem \ref{theorem:pcl}. (The corresponding result for $P^{top}$, building on 
\cite[Theorem 2.3]{baumgartner}, is the main theorem of \cite{ttppfo}.) We also prove a simple lower bound for Ramsey numbers in terms of pigeonhole numbers, which remains our 
best lower bound with the exception of a couple of special cases.

In Section \ref{section:omega+1} we look at $\omega+1$ and prove Theorem \ref{theorem:omega+1}:
\begin{theorem*}
If $k$ is a positive integer, then $R^{top}(\omega+1,k+1)=R^{cl}(\omega+1,k+1)=\omega^k+1$.
\end{theorem*}
In Section \ref{section:steppingup} we describe several ``stepping up'' techniques, resulting in Propositions \ref{proposition:easysteppingup} and \ref{proposition:steppingup2}:
\begin{theorem*}
Let $\alpha$ be a successor ordinal, and let $k$, $m$ and $n$ be positive integers with $k\geq 2$.
\begin{enumerate}
\item
$R^{cl}(\alpha+1,k+1)\leq P^{cl}( R^{cl}(\alpha,k+1), R^{cl}(\alpha+1,k) )+1$.
\item
$R^{cl}(\alpha+1,k+1)\leq P^{cl}(R^{cl}(\alpha,k+1))_k+R^{cl}(\alpha+1,k)$.
\item
$R^{cl}(\omega\cdot m+n+1,k+1)\le R^{cl}(\omega\cdot m+1,k+1)+P^{cl}((R^{cl}(\omega\cdot m+n+1,k))_{2m},R(n,k+1))$.
\end{enumerate}
\end{theorem*}
From these one may deduce upper bounds for $R^{cl}(\omega+n,k)$ for all finite $n,k$. We also indicate a lower bound in Lemma \ref{lemma:omegaplus2lower}, yielding 
$R^{cl}(\omega+2,3)=\omega^2\cdot 2+\omega+2$.

In Section \ref{section:lessthanomegasquared} we use the last of these techniques, which in turn utilizes an ultrafilter argument, to prove Theorem 
\ref{theorem:lessthanomegasquared}:
\begin{theorem*}If $k$ and $m$ are positive integers, then $R^{top}(\omega\cdot m+1,k+1)=R^{cl}(\omega\cdot m+1,k+1)<\omega^\omega$.
\end{theorem*}
In fact, in each case the proof gives an explicit upper bound below $\omega^\omega$.

The ultrafilter approach is then further refined in Section \ref{section:omegasquared} to prove Theorem \ref{theorem:omegasquared}:
\begin{theorem*}
If $k$ is a positive integer, then $R^{top}(\omega^2,k)=R^{cl}(\omega^2,k)\leq\omega^\omega$.
\end{theorem*}
Despite the use of a non-principal ultrafilter on $\omega$, our arguments are formalizable without any appeal to the axiom of choice (see Remark \ref{remark:zf}).

In Section \ref{section:anti-tree} we introduce a different approach in terms of the anti-tree partial ordering on ordinals, which was independently considered by Pi\~na in \cite{pina} for 
her work on extending some results of \cite{baumgartner}. This approach is not used directly in later sections, but does provide a helpful perspective. We use this approach to provide a 
second proof of Theorem \ref{theorem:omega+1}, and to prove Theorem \ref{theorem:topomega2} (but see Remark \ref{remark:omer}):

\begin{theorem*}
$\omega^2\cdot 3\leq R^{top}(\omega\cdot 2,3)\leq\omega^3\cdot 100$.
\end{theorem*}

In Section \ref{section:omegasquared+1} we climb beyond $\omega^2$ and prove Corollary \ref{corollary:omegasquared+1}:
\begin{theorem*}
If $k$ is a positive integer, then $R^{top}(\omega^2+1,k+2)=R^{cl}(\omega^2+1,k+2)\leq\omega^{\omega\cdot k}+1$.
\end{theorem*}
We deduce this from a more general result, which can also be seen as a generalization of Theorem \ref{theorem:omega+1}:
\begin{theorem*}
Let $\alpha$ and $\beta$ be countable ordinals with $\beta>0$, let $k$ be a positive integer, and suppose they satisfy a ``cofinal version'' of
\[\omega^{\omega^\alpha}\to_{cl}(\omega^\beta,k+2)^2.\]
Then
\[\omega^{\omega^\alpha\cdot(k+1)}+1\to_{cl}(\omega^\beta+1,k+2)^2.\]
Moreover, if $\omega^{\omega^\alpha}>\omega^\beta$, then in fact
\[\omega^{\omega^\alpha\cdot k}+1\to_{cl}(\omega^\beta+1,k+2)^2.\]
\end{theorem*}
This is Theorem \ref{theorem:generalomegasquared+1}. (We will give the precise meaning of this cofinal partition relation in the full statement of the theorem.)

Our main result, the weak topological Erd\H{o}s-Milner theorem, is finally proved in Section \ref{section:wtem}. The proof is somewhat technical, but does not directly appeal to any 
earlier results. From this we obtain Corollary \ref{corollary:wtem}:

\begin{theorem*}
Let $\alpha$ be a countable nonzero ordinal, and let $k$, $m$ and $n$ be positive integers.
\begin{enumerate}
\item
$R^{top}(\omega^{\omega^\alpha},k+1)=R^{cl}(\omega^{\omega^\alpha},k+1)\leq\omega^{\omega^{\alpha\cdot k}}$.
\item
$R^{top}(\omega^{\omega^\alpha}+1,k+1)=R^{cl}(\omega^{\omega^\alpha}+1,k+1)\leq
\begin{cases}
\omega^{\omega^{\alpha\cdot k}}+1,&\text{if $\alpha$ is infinite}\\
\omega^{\omega^{(n+1)\cdot k-1}}+1,&\text{if $\alpha=n$ is finite.}
\end{cases}$
\item
$R^{top}(\omega^n\cdot m+1,k+2)=R^{cl}(\omega^n\cdot m+1,k+2)\leq\omega^{\omega^k\cdot n}\cdot R(m,k+2)+1$.
\end{enumerate}
\end{theorem*}

We close in Section \ref{section:questions} with some questions.

\subsection{Acknowledgements}

The first author wants to thank Jean Larson for her enthusiasm and encouragement when this project was just getting started. 
We also want to thank Omer Mermelstein and Stevo Todorcevic, for their 
feedback on an earlier version of this paper and their permission to quote some of their results. 

\section{Preliminaries}\label{section:preliminaries}

In this paper, all arithmetic operations on ordinals are in the ordinal sense. We use interval notation in the usual fashion, so that for example if $\alpha$ and $\beta$ are ordinals then 
$[\alpha,\beta)=\{x\mid\alpha\leq x<\beta\}$. We denote the cardinality of a set $X$ by $|X|$, and use the symbol $\cong$ to denote the homeomorphism relation.

The classical, topological and closed partition relations are defined in the introduction. When $n=1$ we may work with $\beta$ rather than $[\beta]^1$ for simplicity, and we write 
$\beta\to(\alpha)^n_\kappa$ for $\beta\to(\alpha_i)^n_{i\in\kappa}$ when $\alpha_i=\alpha$ for all $i\in\kappa$, and similarly for the topological and closed relations. We refer to the 
function $c$ in these definitions as a \emph{coloring}, and when the range of $c$ is $2$ we identify $0$ with the color red and $1$ with the color blue.

The classical, topological, and closed ordinal pigeonhole and Ramsey numbers are also defined in the introduction. Note that $R((\alpha_i)_{i\in\kappa},(2)_\lambda)=R(\alpha_i)_
{i\in\kappa}$ for any cardinal $\lambda$, and that for fixed $\kappa$, $R(\alpha_i)_{i\in\kappa}$ is a monotonically increasing function of $(\alpha_i)_{i\in\kappa}$ (pointwise), and 
similarly for the topological and closed relations. Note also that the closed partition relation implies the other two, and hence $R(\alpha_i)_{i\in\kappa}\leq R^{cl}(\alpha_i)_{i\in\kappa}$ 
and $R^{top}(\alpha_i)_{i\in\kappa}\leq R^{cl}(\alpha_i)_{i\in\kappa}$ (and similarly for the pigeonhole numbers). 

The topological and closed partition relations are closely related, thanks to the following notion.

\begin{definition}
An ordinal $\alpha$ is said to be \emph{order-reinforcing} if and only if, whenever $X$ is a subspace of an ordinal with $X\cong\alpha$, then there is a subspace $Y\subseteq X$ such 
that $Y$ is order-homeomorphic to $\alpha$.
\end{definition}

It is clear from this definition that, when applied only to order-reinforcing ordinals, the topological and closed pigeonhole and Ramsey numbers coincide. In particular, if $\alpha$ is 
order-reinforcing, then $R^{top}(\alpha,k)=R^{cl}(\alpha,k)$ for all finite $k$. Building on work of Baumgartner \cite[Theorem 0.2]{baumgartner}, the order-reinforcing ordinals were 
classified in \cite[Corollary 2.17]{ttppfo}.

\begin{theorem}
An ordinal $\alpha$ is order-reinforcing if and only if either $\alpha$ is finite, or $\alpha=\omega^\gamma\cdot m+1$ for some nonzero ordinal $\gamma$ and some positive integer 
$m$, or $\alpha=\omega^\gamma$ for some nonzero ordinal $\gamma$.
\end{theorem}

Finally we require the notions of Cantor-Bendixson derivative and rank, which are central to the study of ordinal topologies.

\begin{definition}
Let $X$ be a topological space. The \emph{Cantor-Bendixson derivative} $X^\prime$ of $X$ is the set of limit points of $X$, i.e.,
\[X^\prime=X\setminus\{x\in X\mid x\text{ is isolated}\}.\]

The iterated derivatives of $X$ are defined recursively for $\gamma$ an ordinal by
\begin{enumerate}
\item
$X^{(0)}=X$,
\item
$X^{(\gamma+1)}=\left(X^{(\gamma)}\right)^\prime$, and
\item
$X^{(\gamma)}=\bigcap_{\delta<\gamma}X^{(\delta)}$ when $\gamma$ is a nonzero limit.
\end{enumerate}
\end{definition}

\begin{definition}
If $x$ is a nonzero ordinal, then there is a unique sequence of ordinals $\gamma_1>\gamma_2>\dots>\gamma_n$ and a unique sequence of positive integers $m_1,m_2,\dots,m_n$ 
such that
\[x=\omega^{\gamma_1}\cdot m_1+\omega^{\gamma_2}\cdot m_2+\cdots+\omega^{\gamma_n}\cdot m_n. \]
This representation is the well-known \emph{Cantor normal form} of $x$. The \emph{Cantor-Bendixson rank} of $x$ is defined by
\[\operatorname{CB}(x)=
\begin{cases}
\gamma_n,&\text{if $x>0$}\\
0,&\text{if $x=0$}.
\end{cases}
\]
\end{definition}

The relationship between these two notions is that if $\alpha$ is an ordinal endowed with the order topology and $x\in\alpha$, then the Cantor-Bendixson rank of $x$ is the greatest 
ordinal $\gamma$ such that $x\in\alpha^{(\gamma)}$. Some basic proofs that make use of these notions can be found in the preliminaries of \cite{ttppfo}.

Though we do not need it here, we recall the classification of ordinals up to homeomorphism, due to Flum and Mart\'\i{}nez \cite[2.5 Remark 3]{flummartinez}. Here we state the result 
using the terminology of Kieftenbeld and L\"owe \cite{kieftenbeldlowe}. In order to do this, for a nonzero ordinal $x$ with Cantor normal form
\[x=\omega^{\gamma_1}\cdot m_1+\omega^{\gamma_2}\cdot m_2+\cdots+\omega^{\gamma_n}\cdot m_n,\]
write $\gamma_1(x)$ for $\gamma_1$ and $m_1(x)$ for $m_1$, and let $p(x)$, the \emph{purity} of $x$, be $0$ if $x=\omega^{\gamma_1(x)}\cdot m_1(x)$ and $\gamma_1(x)>0$, and 
$\omega^{\operatorname{CB}(x)}$ otherwise.

\begin{theorem}[Flum-Mart{\'\i}nez]
Two nonzero ordinals $\alpha$ and $\beta$ are homeomorphic if and only if $\gamma_1(\alpha)=\gamma_1(\beta)$, $m_1(\alpha)=m_1(\beta)$, and $p(\alpha)=p(\beta)$.
\end{theorem}

More relevant to this paper is perhaps the classification of ordinals up to \emph{biembeddability} (being homeomorphic to a subspace of one another) rather than homeomorphism. See 
\cite{ttppfo} for further details.

\section{The closed pigeonhole principle for ordinals}\label{section:pcl}

The pigeonhole numbers are an important prerequisite to the study of Ramsey numbers. For example, we have the following lower bound for a Ramsey number given by considering a 
$k$-partite graph.

\begin{proposition}\label{proposition:lowerbound}
If $\alpha\geq 2$ is an ordinal and $k$ is a positive integer, then
\[R(\alpha,k+1)\geq P(\alpha)_k,\]
and similarly for the topological and closed relations: $R^{top}(\alpha,k+1)\ge P^{top}(\alpha,k)$ and $R^{cl}(\alpha,k+1)\ge P^{cl}(\alpha,k)$.
\end{proposition}

\begin{proof}
Suppose $\beta<P(\alpha)_k$. By definition of the pigeonhole number, it follows that there exists a coloring $c:\beta\to k$ such that for all $i\in k$, no subset of $c^{-1}(\{i\})$ is 
order-isomorphic to $\alpha$. To see that $\beta\not\to(\alpha,k+1)^2$, simply consider the coloring $d:[\beta]^2\to\{\text{red,blue}\}$ given by
\[
d(\{x,y\})=
\begin{cases}
\text{red},&\text{if $c(x)=c(y)$}\\
\text{blue},&\text{if $c(x)\neq c(y)$}.
\end{cases}
\]
It is straightforward to verify that $d$ does indeed witness $\beta\not\to(\alpha,k+1)^2$.

For the topological and closed relations, simply replace ``order-isomorphic'' with ``homeomorphic'' or ``order-homeomorphic'' as necessary.
\end{proof}

The values of $P^{top}$ are given in \cite{ttppfo}, and the key result of this section extends this to $P^{cl}$.

\begin{theorem}[The closed pigeonhole principle for ordinals]\label{theorem:pcl}
Given a cardinal $\kappa$ and an ordinal $\alpha_i\geq 2$ for each $i\in\kappa$, there is an algorithm to compute $P^{cl}(\alpha_i)_{i\in\kappa}$. In particular, suppose that $k <
\omega$  and that $\alpha_1,\dots,\alpha_k$ are countable and bigger than one. For any nonzero countable ordinals $\beta_1,\dots,\beta_k$ we have:
\begin{enumerate}
\item
$P^{cl}(\beta_1,\dots,\beta_k,1)=P^{cl}(\beta_1,\dots,\beta_k)$, $P^{cl}(\beta_1)=\beta_1$, $P^{cl}(\beta_1,\dots,\beta_k)$ is invariant under permutations of the $\beta_i$.
\item
$P^{cl}(\omega^{\beta_1}+1,\dots,\omega^{\beta_k}+1)=\omega^{\beta_1\#\cdots\#\beta_k}+1$, where $\#$ denotes the natural (Hessenberg) sum.
\item \label{clause2}
$P^{cl}(\omega^{\beta_1},\dots,\omega^{\beta_k})=\omega^{\beta_1\odot\cdots\odot\beta_k}$, where $\odot$ denotes the Milner-Rado sum. 
\item
If for all $i\le k$, $\beta_i$ is least such that $\alpha_i\le\omega^{\beta_i}$, and if equality holds for at least one $i$, then $P^{cl}(\alpha_1,\dots,\alpha_k)=P^{cl}(\omega^{\beta_1},\dots,
\omega^{\beta_k})$.
\item
If no $\alpha_i$ is a power of $\omega$, write $\alpha_i=1+\gamma_i$ if $\alpha_i$ is finite and $\alpha_i=\omega^{\beta_i}+1+\gamma_i$ otherwise, where $\gamma_i<\omega^
{\beta_i+1}$. Write $t_i=1$ or $t_i=\omega^{\beta_i}+1$ depending on whether $\alpha_i$ is finite. If, for each $i\le k$, we let $Q_i=P^{cl}(\alpha_1,\dots,\alpha_{i-1},\gamma_i,\alpha_
{i+1},\dots,\alpha_k)$, then $P^{cl}(\alpha_1,\dots,\alpha_k)=P^{cl}(t_1,\dots,t_k)+\max\{Q_1,\dots,Q_k\}$.
\end{enumerate}
\end{theorem}

The cases listed cover all uses of the theorem in this paper. To avoid cluttering the paper with an off-topic detour, the reader can take our use of the word ``algorithm'' here merely as a 
shorthand for the full statement of the result, which divides into several additional cases (depending on whether we use infinitely many colors rather than a finite number $k$, or if some 
of the $\alpha_i$ are uncountable). Rather than listing all these cases  explicitly, we refer the reader to \cite{ttppfo} for the full statement for $P^{top}$, and explain within the proof the 
modifications required to obtain $P^{cl}$. We should remark, in particular, that $P^{cl}$ restricted to countable ordinals is a primitive recursive set function and that, given reasonable 
representations of the Cantor normal forms of the countable ordinals $\alpha_1,\dots,\alpha_k$, and of the exponents present in these expressions, the reader should have no 
difficulties providing an actual algorithm that returns (a representation of) the Cantor normal form of $P^{cl}(\alpha_1,\dots,\alpha_k)$.

For completeness, we recall that if $\alpha=\omega^{\gamma_1}\cdot m_1+\dots+\omega^{\gamma_k}\cdot m_k$ and $\beta=\omega^{\gamma_1}\cdot n_1+\dots+
\omega^{\gamma_k}\cdot n_k$, where $\gamma_1>\dots>\gamma_k$ and the $m_i$ and $n_i$ are natural numbers, then $\alpha\#\beta=\omega^{\gamma_1}\cdot(m_1+n_1)+
\dots+\omega^{\gamma_k}\cdot (m_k+n_k)$, and set $0\#\alpha=\alpha\#0=\alpha$ for all $\alpha$. Also, $\alpha\odot\beta$ is the least $\gamma$ such that if 
$\tilde\alpha<\alpha$ and $\tilde\beta<\beta$, then $\tilde\alpha\#\tilde\beta\ne\gamma$. Both of these operations are commutative and associative, allowing brackets to be omitted. 

Before providing the proof of the theorem, let us first look at a few special cases. Some of these will be used later, while others are merely illustrative. The reader may attempt to verify 
them directly.

\begin{fact}\label{fact:pcl}
Let $k$ and $m$ be positive integers.
\begin{enumerate}
\item
$P^{top}(m,k)=P^{cl}(m,k)=P(m,k)=m+k-1$.
\item
$P^{top}(\omega)_2=P^{cl}(\omega)_2=P(\omega)_2=\omega$.
\item
$P^{top}(\omega+1,k)=P^{cl}(\omega+1,k)=\omega\cdot k+1$.
\item
$P^{top}(\omega+1,\omega)=P^{cl}(\omega+1,\omega)=\omega^2$.
\item
$P^{top}(\omega+m,\omega+k)=P^{top}(\omega+1)_2=\omega^2+1$. On the other hand, if $m\geq k$ then $P^{cl}(\omega+m,\omega+k) = \omega^2+\omega\cdot(m-1)+k$.
\item
$P^{top}(\omega\cdot 2)_2=P^{cl}(\omega\cdot 2)_2=\omega^2\cdot 2$. If $m>2$, then $P^{cl}(\omega\cdot m)_2 = \omega^2\cdot(2m-2)$, whereas $P^{top}(\omega\cdot m)_2 = 
\omega^2\cdot(2m-3)+1$.
\item
$P^{top}(\omega\cdot m+k)_2=P^{top}(\omega\cdot m+1)_2=\omega\cdot(2m-1)+1$, while $P^{cl}(\omega\cdot m +k)_2 = \omega^2\cdot(2m-1) + \omega\cdot(k-1) +k$.
\end{enumerate}
\end{fact}

Here is one final special case, for which we provide a short proof.

\begin{proposition}\label{proposition:pomega+1}
If $k$ is a positive integer, then
\[P^{top}(\omega+1)_k=P^{cl}(\omega+1)_k=\omega^k+1.\]
\end{proposition}

(By contrast, $P(\omega+1)_k=\omega\cdot k+1$.)

\begin{proof}
The first equality is immediate from the fact that $\omega+1$ is order-reinforcing.

To see that $\omega^k\not\to_{top}(\omega+1)^1_k$, simply color each $x\in\omega^k$ with color $\operatorname{CB}(x)$, and observe that each color class is discrete.

We prove that $\omega^k+1\to_{top}(\omega+1)^1_k$ by induction on $k$. The case $k=1$ is trivial. For the inductive step, suppose $k\geq 2$ and let $c:\omega^k+1\to k$ be a 
coloring. For each $m\in\omega$, let $X_m=\left[\omega^{k-1}\cdot m+1,\omega^{k-1}\cdot(m+1)\right]\cong\omega^{k-1}+1$. If $c^{-1}(\{i\})\cap X_m=\emptyset$ for some $i\in k$ 
and some $m\in\omega$, then the restriction of $c$ to $X_m$ uses at most $k-1$ colors, and we are done by the inductive hypothesis. Otherwise there exists $x_m\in c^{-1}(\{j\})
\cap X_m$ for each $m\in\omega$ where $j=c(\omega^k)$, whence $\{x_m\mid m\in\omega\}\cup\{\omega^k\}$ is a topological copy of $\omega+1$ in color $j$.
\end{proof}

Here is the proof in full of the closed pigeonhole principle for ordinals.

\begin{proof}[Proof of Theorem \ref{theorem:pcl}]
The argument assumes familiarity with \cite{ttppfo}, where $P^{top}(\alpha_i)_{i\in\kappa}$ is computed. The same proof shows that $P^{cl}(\alpha_i)_{i\in\kappa}=P^{top}(\alpha_i)_
{i\in\kappa}$ except in two cases, which we now examine.

The first case is when $\kappa$ is finite and greater than 1, $\alpha_r\geq\omega_1+1$ for some $r\in\kappa$, $\alpha_i$ is finite for all $i\in\kappa\setminus\{r\}$, and $\alpha_r$ is 
not a power of $\omega$, say $\alpha_r=\omega^\beta\cdot m+1+\gamma$ for some ordinal $\beta$, some positive integer $m$ and some ordinal $\gamma\leq\omega^\beta$ (case 
2(c)(ii) in \cite{ttppfo}). Then
\[P^{top}(\alpha_i)_{i\in\kappa}=\omega^\beta\cdot\left(\sum_{i\in\kappa\setminus\{r\}}(\alpha_i-1)+m\right)+1,\]
whereas a very similar argument shows that
\[P^{cl}(\alpha_i)_{i\in\kappa}=\omega^\beta\cdot\left(\sum_{i\in\kappa\setminus\{r\}}(\alpha_i-1)+m\right)+1+\gamma.\]

The second case is when $(\alpha_i)_{i\in\kappa}$ is a finite sequence of countable ordinals (case 6). In this case the following results carry across for $\beta_1,\beta_2,\dots,\beta_k\in
\omega_1\setminus\{0\}$.
\begin{itemize}
\item
$P^{cl}(\omega^{\beta_1}+1,\omega^{\beta_2}+1,\dots,\omega^{\beta_k}+1)=\omega^{\beta_1\#\beta_2\#\cdots\#\beta_k}+1$, where $\#$ denotes the natural sum 
\cite[Theorem 4.5]{ttppfo}.
\item
$P^{cl}(\omega^{\beta_1},\omega^{\beta_2},\dots,\omega^{\beta_k})=\omega^{\beta_1\odot\beta_2\odot\cdots\odot\beta_k}$, where $\odot$ denotes the Milner-Rado sum 
\cite[Theorem 4.6]{ttppfo}.
\item
Let $\alpha_1,\alpha_2,\dots,\alpha_k\in\omega_1$ with $\alpha_r=\omega^{\beta_r}$ for some $r\in\{1,2,\dots,k\}$. Suppose $\beta_i$ is minimal subject to the condition that $\alpha_i
\leq\omega^{\beta_i}$ for all $i\in\{1,2,\dots,k\}$. We then have that $P^{cl}(\alpha_1,\alpha_2,\dots,\alpha_k)=P^{cl}(\omega^{\beta_1},\omega^{\beta_2},\dots,\omega^{\beta_k})$ by 
\cite[Theorem 4.7]{ttppfo}.
\end{itemize}

Thus it remains to compute $P^{cl}(\alpha_1,\alpha_2,\dots,\alpha_k)$ when $\alpha_i\in\omega_1$ is not a power of $\omega$ for any $i\in\{1,2,\dots,k\}$. In that case, for all $i\in
\{1,2,\dots,k\}$ we may write $\alpha_i=\omega^{\beta_i}+1+\gamma_i$ for some $\beta_i\in\omega_1\setminus\{0\}$ and some ordinal $\gamma_i<\omega^{\beta_i+1}$ (or $\alpha_i=
1+\gamma_i$ for some ordinal $\gamma_i<\omega$ if $\alpha_i$ is finite).

Let $Q_i=P^{cl}(\alpha_1,\dots,\alpha_{i-1},\gamma_i,\alpha_{i+1},\dots,\alpha_k)$ for each $i\in\{1,2,\dots,k\}$. We claim that
\[P^{cl}(\alpha_1,\alpha_2,\dots,\alpha_k)=P^{cl}(\omega^{\beta_1}+1,\omega^{\beta_2}+1,\dots,\omega^{\beta_k}+1)+\operatorname{max}\{Q_1,Q_2,\dots,Q_k\}\]
(replacing $\omega^{\beta_i}+1$ with $1$ if $\alpha_i$ is finite). That this is large enough is clear. That no smaller ordinal is large enough follows from the existence for each 
$r\in\{1,2,\dots,k\}$ of a coloring $c_r:\omega^{\beta_1\#\beta_2\#\cdots\#\beta_k}+1\to\{1,2,\dots,k\}$ with the property that for each $i\in\{1,2,\dots,k\}$, there is a closed copy of 
$\omega^{\beta_i}+1$ in color $i$ if and only if $i=r$, and no closed copy of any ordinal larger than $\omega^{\beta_r}+1$ in color $r$. To obtain such a coloring, simply extend the 
coloring given in \cite[Theorem 4.5]{ttppfo} by setting $c_r(\omega^{\beta_1\#\beta_2\#\cdots\#\beta_k})=r$.
This equation allows $P^{cl}(\alpha_1,\alpha_2,\dots,\alpha_k)$ to be computed recursively, since it will be reduced to the three cases above after a finite number of steps.
\end{proof}

The main idea ultimately leading to Theorem \ref{theorem:pcl} in the setting of countable ordinals is a result of Baumgartner and Weiss \cite[Theorem 2.3]{baumgartner}. Some of the 
details of that argument reappear in several of the proofs below. A particular special case of Theorem \ref{theorem:pcl} of interest is the following easy consequence of that result 
\cite[Corollary 2.4]{baumgartner}.

\begin{corollary}[Baumgartner-Weiss]
If $\beta$ is a nonzero countable ordinal and $k$ is finite, then $P^{cl}(\omega^{\omega^\beta})_k=\omega^{\omega^\beta}$.
\end{corollary}

Although this result follows directly from \cite[Theorem 2.3]{baumgartner}, we briefly illustrate how to obtain it as a special case of our more general result.

\begin{proof}
If $\alpha_i<\omega^\beta$ for all $i\in\{1,\dots,k\}$, then $\alpha_1\#\cdots\#\alpha_k<\omega^\beta$, so that $\omega^\beta\odot\cdots\odot\omega^\beta=\omega^\beta$. Now use 
clause \emph{(\ref{clause2})} of Theorem \ref{theorem:pcl}.
\end{proof}

\section{The ordinal \texorpdfstring{$\omega+1$}{omega plus one}}\label{section:omega+1}

In this section we prove the following result.

\begin{theorem}\label{theorem:omega+1}
If $k$ is a positive integer, then
\[R^{top}(\omega+1,k+1)=R^{cl}(\omega+1,k+1)=\omega^k+1.\]
\end{theorem}

Our proof is quite similar to the proof of Proposition \ref{proposition:pomega+1}. Later, we will provide a second proof (see Section \ref{section:anti-tree}).

\begin{proof}
The first equality is immediate from the fact that $\omega+1$ is order-reinforcing, and the fact that $R^{top}(\omega+1,k+1)\geq\omega^k+1$ follows from Propositions 
\ref{proposition:lowerbound} and \ref{proposition:pomega+1}.

It remains to prove that $\omega^k+1\to_{top}(\omega+1,k+1)^2$, which we do by induction on $k$. The case $k=1$ is trivial. For the inductive step, suppose $k\geq 2$ and let $c:
[\omega^k+1]^2\to \{\text{red},\text{blue}\}$ be a coloring. As in the proof of Proposition \ref{proposition:pomega+1}, for each $m\in\omega$, let $X_m=\left[\omega^{k-1}\cdot 
m+1,\omega^{k-1}\cdot(m+1)\right]\cong\omega^{k-1}+1$. For each $m\in\omega$, we may assume by the inductive hypothesis that $X_m$ has a blue-homogeneous set $B_m$ of 
size $k$, or else it has a red-homogeneous closed copy of $\omega+1$ and we are done. If for any $m\in\omega$ it is the case that $c(\{x,\omega^k\})=\text{blue}$ for all $x\in B_m$, 
then $B_m\cup\{\omega^k\}$ is a blue-homogeneous set of size $k+1$ and we are done. Otherwise for each $m\in\omega$ we may choose $x_m\in B_m$ with $c(\{x_m,\omega^k\})=
\text{red}$. Finally, by Ramsey's theorem, within the set $\{x_m\mid m\in\omega\}$ there is either a blue-homogeneous set of size $k+1$, in which case we are done, or an infinite 
red-homogeneous set $H$, in which case $H\cup\{\omega^k\}$ is a red-homogeneous closed copy of $\omega+1$, and we are done as well.
\end{proof}

It is somewhat surprising that we are able to obtain an exact equality here using the lower bound from Proposition \ref{proposition:lowerbound}. As we shall see, for ordinals larger than 
$\omega+1$ there is typically a large gap between our upper and lower bounds, and we expect that improvement will usually be possible on both sides.

\section{Stepping up by one}\label{section:steppingup}

In this section we consider how to obtain an upper bound for $R^{cl}(\alpha+1,k+1)$, given upper bounds for $R^{cl}(\alpha,k+1)$ and $R^{cl}(\alpha+1,k)$ for a successor ordinal 
$\alpha$ and a positive integer $k$. Note that if $\alpha$ is an infinite successor ordinal, then $\alpha+1\cong\alpha$, so trivially $R^{top}(\alpha+1,k+1)=R^{top}(\alpha,k+1)$, which is 
why we study $R^{cl}$ instead.

First we give two simple upper bounds. The first comes from considering the edges incident to the largest point, as in a standard proof of the existence of the finite Ramsey numbers. 
The second typically gives worse bounds, but is more similar to the technique we will consider next.

\begin{proposition}\label{proposition:easysteppingup}
Let $\alpha$ be a successor ordinal, and let $k\geq 2$ be a positive integer.
\begin{enumerate}
\item
$R^{cl}(\alpha+1,k+1)\leq P^{cl}( R^{cl}(\alpha,k+1), R^{cl}(\alpha+1,k) )+1$.
\item
$R^{cl}(\alpha+1,k+1)\leq P^{cl}(R^{cl}(\alpha,k+1))_k+R^{cl}(\alpha+1,k)$.
\end{enumerate}
\end{proposition}

\begin{proof}
First note that since $\alpha$ is a successor ordinal, a closed copy of $\alpha+1$ is obtained from any closed copy of $\alpha$ together with any larger point.
\begin{enumerate}
\item
Let $\beta=P^{cl}(R^{cl}(\alpha,k+1),R^{cl}(\alpha+1,k))$ and let $c:[\beta+1]^2\to\{\text{red},\text{blue}\}$ be a coloring. This induces a coloring $d:\beta\to\{\text{red},\text{blue}\}$ given 
by $d(x)=c(\{x,\beta\})$. By definition of $P^{cl}$, there exists $X\subseteq\beta$ such that either $X\subseteq d^{-1}(\{\text{red}\})$ and $X$ is a closed copy of $R^{cl}(\alpha,k+1)$, or 
$X\subseteq d^{-1}(\{\text{blue}\})$ and $X$ is a closed copy of $R^{cl}(\alpha+1,k)$. In the first case, by definition of $R^{cl}$, we are either done immediately, or we obtain a 
red-homogeneous closed copy $Y$ of $\alpha$, in which case $Y\cup\{\beta\}$ is a red-homogeneous closed copy of $\alpha+1$. The second case is similar.
\item
Let $\beta=P^{cl}(R^{cl}(\alpha,k+1))_k$ and let $c:[\beta+R^{cl}(\alpha+1,k)]^2\to\{\text{red},\text{blue}\}$ be a coloring. By definition of $R^{cl}$, either we are done or there is a 
blue-homogeneous set of $k$ points, say $\{x_1,x_2,\dots,x_k\}$, with $x_i\ge\beta$ for all $i\in\{1,\dots,k\}$. If for any $y<\beta$ we have $c(\{y,x_i\})=\text{blue}$ for all $i$, then we 
are done. Otherwise define a coloring $d:\beta\to k$ by taking $d(y)$ to be some $i$ such that $c(\{y,x_i\})=\text{red}$. Then by definition of $P^{cl}$, there exists $i$ and $X\subseteq
\beta$ such that $X\subseteq d^{-1}(\{i\})$ and $X$ is a closed copy of $R^{cl}(\alpha,k+1)$. But then $X$ either contains a blue-homogeneous set of $k+1$ points, or a 
red-homogeneous closed copy $Y$ of $\alpha$, in which case $Y\cup\{x_i\}$ is a red-homogeneous closed copy of $\alpha+1$, as required.\qedhere
\end{enumerate}
\end{proof}

Note that $R^{cl}(\alpha,2)=\alpha$ for any ordinal $\alpha$. Hence if $\alpha$ is a successor ordinal and we have an upper bound on $R^{cl}(\alpha,k)$ for every positive integer $k$, 
then we may easily apply either of these two inequalities recursively to obtain upper bounds on $R^{cl}(\alpha+n,k)$ for all positive integers $n$ and $k$.

Unfortunately, neither of these techniques appears to generalize well to limit ordinals, since they are ``backward-looking'' in some sense. Our ``forward-looking'' technique is a little more 
complicated and, in the form presented here, it only works below $\omega^2$. Before we state the general result, here is an illustrative special case.

\begin{lemma}\label{lemma:omega+2}
$R^{cl}(\omega+2,3)\leq R^{cl}(\omega+1,3)+P^{cl}(\omega+2)_2=\omega^2\cdot 2+\omega+2$.
\end{lemma}

\begin{proof}
First note that $R^{cl}(\omega+1,3)+P^{cl}(\omega+2)_2=(\omega^2+1)+(\omega^2+\omega+2)=\omega^2\cdot 2 +\omega+2$ by Theorem \ref{theorem:omega+1} and Fact 
\ref{fact:pcl}. It remains to prove that $R^{cl}(\omega+1,3)+P^{cl}(\omega+2)_2\to_{cl}(\omega+2,3)^2$.

Let $\beta=R^{cl}(\omega+1,3)$, let $c:[\beta+P^{cl}(\omega+2)_2]^2\to\{\text{red},\text{blue}\}$ be a coloring, and suppose for contradiction that there is no red-homogeneous closed 
copy of $\omega+2$ and no blue triangle.

By definition of $R^{cl}$ and our assumption that the coloring admits no blue triangles, there exists a red-homogeneous closed copy $X$ of $\omega+1$ contained in $\beta$. Let $x$ 
be the largest point in $X$ and let $H=X\setminus\{x\}$.

Let $A_1=\{y\geq\beta\mid c(\{x,y\})=\text{blue}\}$ and let $A_2=\{y\geq\beta\mid c(\{h,y\})=\text{blue}$ for all but finitely many $h\in H\}$.

First of all, we claim that $A_1$ is red-homogeneous. This is because if $y,z\in A_1$ and $c(\{y,z\})=\text{blue}$, then $\{x,y,z\}$ is a blue triangle.

Next, we claim that $A_2$ is red-homogeneous. This is because if $y,z\in A_2$ and $c(\{y,z\})=\text{blue}$, then $c(\{h,y\})=c(\{h,z\})=\text{blue}$ for all but finitely many $h\in H$, and 
so $\{h,y,z\}$ is a blue triangle for any such $h$.

Finally, we claim that if $y\geq\beta$, then $y\in A_1\cup A_2$. For otherwise we would have $y\geq\beta$ with $c(\{x,y\})=\text{red}$ and $c(\{h,y\})=\text{red}$ for all $h$ in some 
infinite subset $K\subseteq H$, whence $K\cup\{x,y\}$ is a red-homogeneous closed copy of $\omega+2$. Hence by definition of $P^{cl}$, either $A_1$ or $A_2$ contains a closed 
copy of $\omega+2$, which by the above claims must be red-homogeneous, and we are done.
\end{proof}

Digressing briefly, we show that in this particular case the upper bound is optimal, and thus $R^{cl}(\omega+2,3)=\omega^2\cdot 2+\omega+2$. The coloring we present was found by 
Omer Mermelstein.

\begin{lemma}\label{lemma:omegaplus2lower}
$R^{cl}(\omega+2,3)\geq\omega^2\cdot 2+\omega+2$.
\end{lemma}

\begin{proof}
We provide a coloring witnessing $\omega^2\cdot 2+\omega+1\not\to_{cl}(\omega+2,3)^2$. In order to define it, let $G$ be the graph represented by the following diagram.
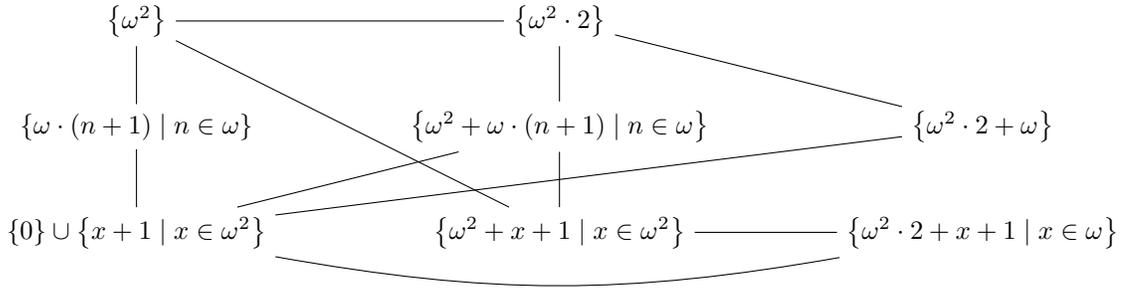
\begin{center}
\begin{tikzpicture}[x=80,y=40]
\node(bottom1)at(-2,-1){$\left\{0\right\}\cup\left\{x+1\mid x\in\omega^2\right\}$};
\node(middle1)at(-2,0){$\left\{\omega\cdot(n+1)\mid n\in\omega\right\}$};
\node(top1)at(-2,1){$\left\{\omega^2\right\}$};
\node(bottom2)at(0,-1){$\left\{\omega^2+x+1\mid x\in\omega^2\right\}$};
\node(middle2)at(0,0){$\left\{\omega^2+\omega\cdot (n+1)\mid n\in\omega\right\}$};
\node(top2)at(0,1){$\left\{\omega^2\cdot 2\right\}$};
\node(bottom3)at(2,-1){$\left\{\omega^2\cdot 2+x+1\mid x\in\omega\right\}$};
\node(middle3)at(2,0){$\left\{\omega^2\cdot 2+\omega\right\}$};
\path(top1)edge(middle1);
\path(middle1)edge(bottom1);
\path(top1)edge(top2);
\path(top1)edge(bottom2);
\path(bottom1)edge(middle2);
\path(top2)edge(middle2);
\path(middle2)edge(bottom2);
\path(bottom1)edge(middle3);
\path(top2)edge(middle3);
\path(bottom2)edge(bottom3);
\path(bottom1)edge[bend right=10](bottom3);
\end{tikzpicture}
\end{center}
Define a coloring $c:[\omega^2\cdot 2+\omega+1]^2\to\{\text{red},\text{blue}\}$ by setting $c(\{x,y\})=\text{blue}$ if and only if $x$ and $y$ lie in distinct, adjacent vertices of $G$. First 
note that there is no blue triangle since $G$ is triangle-free. Now suppose $X$ is a closed copy of $\omega+2$, and write $X=Z\cup\{x\}\cup\{y\}$ with $z<x<y$ for all $z\in Z$. If $Z\cup
\{x\}$ is red-homogeneous, then either $x=\omega^2$ and $Z\subseteq\{0\}\cup\{x+1\mid x\in\omega^2\}$, or $x=\omega^2\cdot 2$ and (discarding a finite initial segment of $Z$ if 
necessary) we may assume $Z\subseteq\{\omega^2+x+1\mid x\in \omega^2\}$. In each case either $c(\{x,y\})=\text{blue}$ or $c(\{z,y\})=\text{blue}$ for all $z\in Z$. Hence $X$ cannot 
be red-homogeneous, and we are done.
\end{proof}

Here is the general formulation of our ``forward-looking'' technique.

\begin{proposition}\label{proposition:steppingup}
Let $k$, $m$ and $n$ be positive integers with $k\geq 2$. Then
\[R^{cl}(\omega\cdot m+n+1,k+1)\leq R^{cl}(\omega\cdot m+n,k+1)+P^{cl}(R^{cl}(\omega\cdot m+n+1,k))_{2m+n-1}.\]
\end{proposition}

\begin{proof}
Let $\beta=R^{cl}(\omega\cdot m+n,k+1)$ and let $c:[\beta+P^{cl}(R^{cl}(\omega\cdot m+n+1,k))_{2m+n-1}]^2\to\{\text{red},\text{blue}\}$ be a coloring.

By definition of $R^{cl}$, either there is a blue-homogeneous set of $k+1$ points, in which case we are done, or there exists $X\subseteq\beta$ such that $X$ is a red-homogeneous 
closed copy of $\omega\cdot m+n$. In that case, write
\[X=H_1\cup\{x_1\}\cup H_2\cup\{x_2\}\cup\dots\cup\{x_{m-1}\}\cup H_m\cup\{y_1,y_2,\dots,y_n\}\]
with $H_1,H_2,\dots,H_m$ each of order type $\omega$ and $h_1<x_1<h_2<x_2<\dots<x_{m-1}<h_m<y_1<y_2<\dots<y_n$ whenever $h_i\in H_i$ for all $i\in\{1,2,\dots m\}$.

For each $z\geq\beta$, if every single one of the following $2m+n-1$ conditions holds, then $X\cup\{z\}$ contains a closed copy of $\omega\cdot m+n+1$, and we are done.
\begin{itemize}
\item
$c(\{h,z\})=\text{red}$ for infinitely many $h\in H_i$ (one condition for each $i\in\{1,2,\dots m\}$)
\item
$c(\{x_i,z\})=\text{red}$ (one condition for each $i\in\{1,2,\dots,m-1\}$)
\item
$c(\{y_i,z\})=\text{red}$ (one condition for each $i\in\{1,2,\dots,n\}$)
\end{itemize}

Thus we may assume that one of these conditions fails for each $z\geq\beta$. This induces a $2m+n-1$-coloring of these points, so by definition of $P^{cl}$ there is a closed copy $X$ 
of $R^{cl}(\omega\cdot m+n+1,k)$ among these points such that the same condition fails for each $z\in X$.

Finally, by definition of $R^{cl}$, $X$ contains either a red-homogeneous closed copy of $\omega\cdot m+n+1$, in which case we are done, or a blue-homogeneous set of $k$ points, in 
which case using the failed condition we can find a final point with which to construct a blue-homogeneous set of $k+1$ points.
\end{proof}

We now present a refinement of Proposition \ref{proposition:steppingup} that introduces ideas we will be elaborating on in the following sections. This refinement allows us to step up 
from $\omega\cdot m+1$ to $\omega\cdot m+n+1$ in one step, which in turn translates into improved bounds. 

\begin{proposition} \label{proposition:steppingup2}
For all positive integers $k,m,n$ with $k\ge 2$, we have 
 $$ R^{cl}(\omega\cdot m+n+1,k+1) \leq R^{cl}(\omega\cdot m+1,k+1) + P^{cl}( (R^{cl}(\omega\cdot m+n+1,k))_{2m}, R(n,k+1) ), $$
where the notation $((x)_{2m},y)$ is shorthand for $(x,x,\dots,x,y)$ with $x$ appearing precisely 2m times.
\end{proposition}

Proposition \ref{proposition:steppingup2} formally supersedes Proposition \ref{proposition:steppingup}. We have decided to include both versions since the combinatorics are somewhat 
simpler in the earlier argument. The new proof uses a non-principal ultrafilter on $\omega$. It turns out that this use is superfluous, but makes the idea more transparent, since in some 
sense the ultrafilter makes many choices for us. The result is actually provable in (a weak sub-theory of) $\mathsf{ZF}$ (see Remark \ref{remark:zf}).

\begin{proof}
Fix a non-principal ultrafilter $\mathcal U$ on $\omega$, let 
 $$ \gamma=R^{cl}(\omega\cdot m+1,k+1)\quad\mbox{ and }\quad\beta= P^{cl}( (R^{cl}(\omega\cdot m+n+1,k))_{2m}, R(n,k+1) ), $$ 
and consider a coloring $c:[\gamma+\beta]^2\to\{\mbox{red,blue}\}$. By definition of $\gamma$, we may assume that there is a set $H=H_1\cup\{x_1\}\cup\dots\cup H_m\cup\{x_m\}
\subseteq\gamma$ that is a red-homogenous closed copy of $\omega\cdot m+1$ (else, there is a blue-homogeneous subset of $\gamma$ of $k+1$ points, and we are done). Here, 
each $H_i$ has type $\omega$, and $a<x_i<b<x_{i+1}$ whenever $i<m$, $a\in H_i$ and $b\in H_{i+1}$. List the elements of each $H_i$ in increasing order as $h_{i,0}<h_{i,1}<
\dots$\,.

Now every ordinal in $\gamma+\beta$ larger than $x_m$ lies in at least one of the following $2m+1$ sets:
\begin{itemize}
\item 
$A_i=\{a>x_m\mid c(x_i,a)=\mbox{blue}\}$, for each $i\le m$,
\item   
$B_i=\{a>x_m\mid\{n\mid c(h_{i,n},a)=\mbox{blue}\}\in{\mathcal U}\}$, for each $i\le m$, and 
\item
$C=\{a>x_m\mid c(x_j,a)=\mbox{red and }\{n\mid c(h_{j,n},a)=\mbox{red}\}\in{\mathcal U}\mbox{ for all }j\le m\}$.
\end{itemize}

Hence, by definition of $\beta$, one of the following must hold.
\begin{enumerate}
\item
Some $A_i$ contains a closed copy of $R^{cl}(\omega\cdot m+n+1,k)$. In turn, this copy either contains a red-homogeneous closed copy of $\omega\cdot m+n+1$, and we are done, 
or a blue-homogeneous set $D$ of $k$ points, in which case $\{x_i\}\cup D$ is blue-homogeneous of size $k+1$, and we are done as well.
\item
Some $B_i$ contains a closed copy of $R^{cl}(\omega\cdot m+n+1,k)$. This copy either contains a red-homogeneous closed copy of $\omega\cdot m+n+1$, and we are done, or a 
blue-homogeneous set $D$ of $k$ points, in which case $\{x\}\cup D$ is blue-homogeneous of size $k+1$ for any $x\in\{h\in H_i\mid c(h,a)=\mbox{blue}$ for all $a\in D\}$, and we are 
done as well. Here we are using that the intersection of finitely many sets in $\mathcal U$ is non-empty.
\item
$C$ contains a set of size $R(n,k+1)$. This set either contains a red-homogeneous set $D$ of size $n$, in which case for each $i$ there is an infinite subset $H'_i$ of $H_i$ such that 
$H'_i\cup D$ is red-homogeneous, and therefore
 $$ H_0'\cup\{x_0\}\cup\cdots\cup H_m'\cup\{x_m\}\cup D $$ 
is a red-homogeneous closed copy of $\omega\cdot m+n+1$, and we are done, or else it contains a blue-homogeneous set $E$ of size $k+1$, in which case we are done as well. Here 
we are using that the intersection of finitely many sets in $\mathcal U$ is actually infinite. 
\end{enumerate}

The result follows.
\end{proof}

Because of Theorem \ref{theorem:omega+1}, any one of the inequalities from Propositions \ref{proposition:easysteppingup}, \ref{proposition:steppingup} and 
\ref{proposition:steppingup2} is enough for us to obtain upper bounds for $R^{cl}(\omega+n,k)$ for all finite $n,k$. We conclude this section with an explicit statement of a few of these 
upper bounds. Curiously, in the second part, we require both Propositions \ref{proposition:easysteppingup} and \ref{proposition:steppingup2} in order to obtain the best bound.

\begin{corollary}
\begin{enumerate}
\item\label{item:omega+2,4}
$R^{cl}(\omega+2,4)\leq\omega^4\cdot 3+\omega^3+\omega^2+\omega+2$.
\item
If $k\geq 4$ is a positive integer, then
\[R^{cl}(\omega+2,k+1)\leq\omega^r\cdot 3+\omega^{r-1}+\omega^{r-2}+\omega^{r-3}+\omega^{r-4}+\omega^{r-8}+\omega^{r-13}+\omega^{r-19}+\dots+\omega^k+2,\]
where $r=\frac{k^2+k-4}2$.
\item\label{item:omega+n,3}
If $n$ is a positive integer, then $R^{cl}(\omega+n,3)<\omega^3$. In particular, 
\[R^{cl}(\omega+3,3)\leq \omega^2\cdot 4+\omega\cdot 2+3.\]
\end{enumerate}
\end{corollary}

\begin{proof}
\begin{enumerate}
\item
By Proposition \ref{proposition:steppingup2} and Lemma \ref{lemma:omega+2}, $R^{cl}(\omega+2,4)\leq R^{cl}(\omega+1,4)+P^{cl}((R^{cl}(\omega+2,3))_2,R(1,4))=\omega^3+1+
P^{cl}(\omega^2\cdot 2+\omega+2)_2=\omega^4\cdot 3+\omega^3+\omega^2+\omega+2$.
\item
This can be obtained from part \ref{item:omega+2,4} by recursively applying the first inequality from Proposition \ref{proposition:easysteppingup}.
\item
By Proposition \ref{proposition:steppingup2}, $R^{cl}(\omega+n+1,3)\leq \omega^2+1+P^{cl}((\omega+n+1)_2,R(n,3))$. Using Lemma \ref{lemma:omega+2} for the base case, it is easy 
to see by induction that the second term here is below $\omega^3$, but then so is the full term. The particular case can be verified by direct computation.\qedhere
\end{enumerate}
\end{proof}

For comparison, the corresponding lower bounds given by Proposition \ref{proposition:lowerbound} are as follows. If $k,n\geq 3$ are positive integers, then
\[R^{cl}(\omega+2,k+1)\geq P^{cl}(\omega+2)_k=\omega^k+\omega^{k-1}+\dots+\omega+2\]
and
\[R^{cl}(\omega+n,3)\geq P^{cl}(\omega+n)_2=\omega^2+\omega\cdot(n-1)+n,\]
but note the latter is already far behind $R^{cl}(\omega+n,3)\ge R^{cl}(\omega+2,3)=\omega^2\cdot 2+\omega+2$.

\section{Ordinals less than \texorpdfstring{$\omega^2$}{omega squared}}\label{section:lessthanomegasquared}

So far we only have upper bounds on $R^{cl}(\alpha,k)$ for $\alpha<\omega\cdot 2$. In this section we extend this to $\alpha<\omega^2$ with the following result.

\begin{theorem}\label{theorem:lessthanomegasquared}
If $k$ and $m$ are positive integers, then
\[R^{cl}(\omega\cdot m+1,k+1)<\omega^\omega.\]
\end{theorem}

In fact, in each case the proof gives an explicit upper bound below $\omega^\omega$.

Our proof builds upon the stepping up technique from Proposition \ref{proposition:steppingup2}. We also make use of some classical ordinal Ramsey theory, namely, the following 
result.

\begin{theorem}[Erd\H{o}s-Rado]\label{theorem:ordinarylessthanomegasquared}
If $k$ and $m$ are positive integers, then $R(\omega\cdot m,k)<\omega^2$.
\end{theorem}

In fact, Erd\H{o}s and Rado computed the exact values of these Ramsey numbers in terms of a combinatorial property of finite digraphs. More precisely, we consider digraphs for which 
loops are not allowed, but edges between two vertices pointing in both directions are allowed. The complete digraph on $m$ vertices is denoted by $K_m^*$. Recall that a 
\emph{tournament} of order $k$ is a digraph obtained by assigning directions to the edges of the complete (undirected) graph on $k$ vertices, and that a tournament is transitive if and 
only if these assignments are compatible, that is, if and only if whenever $x$, $y$ and $z$ are distinct vertices with an edge from $x$ to $y$ and an edge from $y$ to $z$, then there is 
also an edge from $x$ to $z$. The class of transitive tournaments of order $k$ is denoted by $L_k$.

Using this terminology, Theorem \ref{theorem:ordinarylessthanomegasquared} can be deduced from the following two results of Erd\H{o}s and Rado (who stated them in 
a slightly different manner), see \cite[Theorem 25]{erdosrado} and \cite{erdosradot}.

\begin{lemma}[Erd\H{o}s-Rado]
If $k$ and $m$ are positive integers, then there is a positive integer $p$ such that any digraph on $p$ or more vertices admits either an independent set of size $m$, or a transitive 
tournament of order $k$. We denote the least such $p$ by $R(K_m^*,L_k)$.
\end{lemma}

\begin{theorem}[Erd\H{o}s-Rado]\label{theorem:ordinarylessthanomegasquaredprecise}
If $k,m>1$ are positive integers, then $R(\omega\cdot m,k)=\omega\cdot R(K_m^*,L_k)$.
\end{theorem}

Before indicating how to prove Theorem \ref{theorem:lessthanomegasquared} in general, we first illustrate the key ideas with the following special case. We use the special case of 
Theorem \ref{theorem:ordinarylessthanomegasquaredprecise} that $R(\omega\cdot 2,3)=\omega\cdot 4$.

\begin{lemma}\label{lemma:omega2+1}
$R^{cl}(\omega\cdot 2+1,3)\leq\omega^8\cdot 7+1.$
\end{lemma}

As with Proposition \ref{proposition:steppingup2}, the proof uses a non-principal ultrafilter on $\omega$, but see Remark \ref{remark:zf}. 

\begin{proof}
To make explicit the reason for using the ordinal $\omega^8\cdot 7+1$, let
\begin{itemize}
\item
$\beta_1 = P^{cl}(\omega\cdot 2+1,\omega\cdot 2+1,\omega^2 +1)=\omega^4\cdot 3+1$,
\item
$\beta_2 = P^{cl}(\omega\cdot 2+1,\omega\cdot 2+1,\omega^2+1+\beta_1)=\omega^6\cdot 5+1$,
\item
$\beta_3 = P^{cl}(\omega\cdot 2+1,\omega\cdot 2+1,\omega^2+1+\beta_2)=\omega^8\cdot 7+1$ and
\item
$\beta=\omega^2+1+\beta_3=\omega^8\cdot 7+1$.
\end{itemize}
Fix a non-principal ultrafilter $\mathcal U$ on $\omega$ and let $c:[\beta]^2\to\{\text{red},\text{blue}\}$ be a coloring.

Among the first $\omega^2+1$ elements of $\beta$, we may assume that we have a red-homogeneous closed copy of $\omega+1$. Let $x$ be its largest point and let $H$ be its
subset of order type $\omega$. Write $H=\{h_n\mid n\in\omega\}$ with $h_0<h_1<\dots$.

Now the points in $\beta$ larger than $x$ form a disjoint union $A_1\cup A_2\cup A_3$, where
\begin{itemize}
\item
$A_1=\{a>x\mid c(\{x,a\})=\text{blue}\}$,
\item
$A_2=\{a>x\mid c(\{x,a\})=\text{red}$ but $\{n\in\omega\mid c(\{h_n,a\})=\text{blue}\}\in\mathcal U\}$ and
\item
$A_3=\{a>x\mid c(\{x,a\})=\text{red}$ and $\{n\in\omega\mid c(\{h_n,a\})=\text{red}\}\in\mathcal U\}$.
\end{itemize}

If either $A_1$ or $A_2$ contains a closed copy of $\omega\cdot 2+1$, then we are done. (For $A_2$, we use the fact that if $U,V\in\mathcal U$ then $U\cap V\in\mathcal U$.) So by 
definition of $\beta_3$, we may assume that $A_3$ contains a closed copy $X$ of $\omega^2+1+\beta_2$.

Now we repeat the argument within $X$. Among the first $\omega^2+1$ members of $X$, we may assume we that have a red-homogeneous closed copy of $\omega+1$. Let $y$ be its 
largest point, and let $I$ be its subset of order type $\omega$. Write $I=\{i_n\mid n\in\omega\dots\}$ with $i_0<i_1<\dots$. (Note that at this stage, $\{n\in\omega\mid c(\{h_n,i\})=
\text{red}\}\in\mathcal U$ for all $i\in I$, yet we cannot conclude from this that there are infinite subsets $H^\prime\subseteq H$ and $I^\prime\subseteq I$ such that $H^\prime\cup 
I^\prime$ is red-homogeneous.)

Just as before, write the subset of $X$ lying above $y$ as a disjoint union $B_1\cup B_2\cup B_3$, where $b\in B_1$ if and only if $c(\{y,b\})=\text{blue}$, $b\in B_2$ if and only if 
$c(\{y,b\})=\text{red}$ but $\{n\in\omega\mid c(\{i_n,b\})=\text{blue}\}\in\mathcal U$, and $b\in B_3$ if and only if $c(\{y,b\})=\text{red}$ and $\{n\in\omega\mid c(\{i_n,b\})=\text{red}\}\in
\mathcal U$. Again we may conclude from the definition of $\beta_2$ that $B_3$ contains a closed copy $Y$ of $\omega^2+1+\beta_1$.

Repeat this argument once again within $Y$, and then pass to a final red-homogeneous closed copy of $\omega+1$. We obtain a closed set
\[H\cup\{x\}\cup I\cup\{y\}\cup J\cup\{z\}\cup K\cup\{w\}\]
of order type $\omega\cdot 4+1$ (with $J=\{j_n\mid n\in\omega\}$, $j_0<j_1<\dots$ and $K=\{k_n\mid n\in\omega\}$, $k_0<k_1<\dots$) such that $H\cup\{x\}$, $I\cup\{y\}$, $J\cup\{z\}$ 
and $K\cup\{w\}$ are red-homogeneous, $\{x,y,z,w\}$ is red-homogeneous, $c(\{x,i_n\})=c(\{x,j_n\})=c(\{x,k_n\})=c(\{y,j_n\})=c(\{y,k_n\})=c(\{z,k_n\})=\text{red}$ for all $n\in\omega$, and 
finally for any $a>x$, $b>y$ and $c>z$ in this set, we have $\{n\in\omega\mid c(\{h_n,a\})=\text{red}\}\in\mathcal U$, $\{n\in\omega\mid c(\{i_n,b\})=\text{red}\}\in\mathcal U$ and 
$\{n\in\omega\mid c(\{j_n,c\})=\text{red}\}\in\mathcal U$.

At last we use the ultrafilter, in which is the crucial step of the argument. Let
\begin{itemize}
\item
$H^\prime = \{h\in H\mid c(\{h,y\})=c(\{h,z\})=c(\{h,w\})=\text{red}\}$,
\item
$I^\prime = \{i\in I\mid c(\{i,z\})=c(\{i,w\})=\text{red}\}$ and
\item
$J^\prime = \{j\in J\mid c(\{j,w\})=\text{red}\}$,
\end{itemize}
each of which corresponds to some $U\in\mathcal U$ and is therefore infinite. (Note that if we had tried to argue directly, without using the ultrafilter or modifying the construction in a 
substantial way, then we would have been able to deduce that $J^\prime$ is infinite, but it would not have been apparent that $H^\prime$ or $I^\prime$ are.) This ensures that 
$c(\{a,b\})=\text{red}$ whenever $a\in\{x,y,z,w\}$ and $b\in H^\prime\cup\{x\}\cup I^\prime\cup\{y\}\cup J^\prime\cup\{z\}\cup K\cup\{w\}$.

To complete the proof, recall that that $\omega\cdot 4 \to (\omega\cdot 2,3)^2$. It follows that there is either a blue triangle, in which case we are done, or a red-homogeneous subset 
$M\subseteq H^\prime\cup I^\prime\cup J^\prime\cup K$ of order type $\omega\cdot 2$. Let $S$ be the initial segment of $M$ of order type $\omega$ and $T=M\setminus S$, and let 
$s=\operatorname{sup}(S)$ and $t=\operatorname{sup}(T)$ (so that $s,t\in\{x,y,z,w\}$). Then $S\cup\{s\}\cup T\cup\{t\}$ is a red-homogeneous closed copy of $\omega\cdot 2+1$, and 
we are done.
\end{proof}

This argument easily adapts to give the following.

\begin{proposition} \label{remark:omega8}
$R^{cl}(\omega\cdot 2,3)\le\omega^4\cdot 2$.
\end{proposition}

\begin{proof}
The proof follows the argument for Lemma \ref{lemma:omega2+1}, but is simpler. The point is that at each step of the process we only need to split into two rather than three sets, and 
there is no need to use ultrafilters. In more detail, let
\begin{itemize}
\item
$\beta_1=P^{cl}(\omega\cdot 2,\omega)=\omega^2$,
\item
$\beta_2=P^{cl}(\omega\cdot 2,\omega^2+1+\beta_1)=\omega^3\cdot 2$,
\item 
$\beta_3=P^{cl}(\omega\cdot 2,\omega^2+1+\beta_2)=\omega^4\cdot 2$, and
\item
$\beta=\omega^2+1+\beta_3=\omega^4\cdot 2$.
\end{itemize}
Fix a coloring $c:[\beta]^2\to\{\mbox{red,blue}\}$. Among the first $\omega^2+1$ elements of $\beta$ we may assume we have a red-homogeneous closed copy $H_1\cup\{x_1\}$ of 
$\omega+1$, with $h<x$ for all $h\in H_1$. Split the ordinals in $\beta$ above $x$ into two sets:
\begin{itemize}
\item
$A_1=\{a>x\mid c(x,a)=\mbox{blue}\}$, and 
\item
$A_2=\{a>x\mid c(x,a)=\mbox{red}\}$.
\end{itemize}
As usual, we may assume $A_1$ is red-homogeneous and does not contain a closed copy of $\omega\cdot 2$, so that $A_2$ contains a closed copy $X$ of $\omega^2+1+\beta_2$. 
As in the proof of Lemma \ref{lemma:omega2+1}, iterating the argument eventually produces a closed copy 
 $$ H=H_1\cup\{x_1\}\cup H_2\cup\{x_2\}\cup H_3\cup \{x_3\}\cup H_4 $$
of $\omega\cdot 4$ with each $H_i$ of type $\omega$ and $h_1<x_1<h_2<x_2<h_3<x_3<h_4$ for all $h_i\in H_i$, $i\in\{1,2,3,4\}$, such that $H_i\cup\{x_i\}$ is red-homogeneous for 
all $i\le 3$, and $c(x_i,h)=\mbox{red}$ for all $i\le 3$ and all $h>x_i$. Now use that $\omega\cdot 4\to(\omega\cdot 2,3)^2$ to find either a blue-homogeneous triangle in $H$, or some 
$i<j$ in $\{1,2,3,4\}$, and infinite sets $H_i'\subseteq H_i$, $H_j'\subseteq H_j$ such that $H_i'\cup H_j'$ is a red-homogeneous copy of $\omega\cdot 2$, in which case $H_i'\cup
\{x_i\}\cup H_j'$ is a red-homogeneous closed copy of $\omega\cdot 2$.
\end{proof}

We now indicate the modifications to the argument of Lemma \ref{lemma:omega2+1} that are required to obtain the general result.

\begin{proof}[Proof of Theorem \ref{theorem:lessthanomegasquared}]
The proof is by induction on $k$. The case $k=1$ is trivial. For the inductive step, suppose $k\geq 2$. We can now use the argument of Lemma \ref{lemma:omega2+1} with just a 
couple of changes.

Firstly, we require $\omega^k+1$ points in order to be able to assume that we have a red-homogeneous closed copy of $\omega+1$.

Secondly, it is no longer enough for $A_1$ or $A_2$ to contain a closed copy of $\omega\cdot m+1$, but it is enough for one of them to contain a closed copy of $R^{cl}(\omega\cdot 
m+1,k)$, which we have an upper bound on by the inductive hypothesis (and likewise for $B_1$ and $B_2$, and so on).

Finally, in order to complete the proof using Theorem \ref{theorem:ordinarylessthanomegasquaredprecise}, we must iterate the argument $R(K^\ast_m,L_{k+1})$ times.

This argument demonstrates that $R^{cl}(\omega\cdot m+1,k+1)\leq\omega^k+1+\beta_{R(K^\ast_m,L_{k+1})-1}$, where $\beta_0=0$ and $\beta_i=P^{cl}(R^{cl}(\omega\cdot 
m+1,k),R^{cl}(\omega\cdot m+1,k),\omega^k+1+\beta_{i-1})$ for $i\in\{1,2,\dots,R(K^\ast_m,L_{k+1})-1\}$.
\end{proof}

To obtain upper bounds for $R^{cl}(\omega\cdot m+n,k)$ for all finite $k$, $m$ and $n$, one can again use any of the three inequalities from Propositions 
\ref{proposition:easysteppingup} 
and \ref{proposition:steppingup2}, which may give better bounds than simply using the bound on $R^{cl}(\omega\cdot(m+1)+1,k)$ given by Theorem 
\ref{theorem:lessthanomegasquared}.

\begin{remark}
It is perhaps worth pointing out that the classical version of this problem, the precise computation of the numbers $R(\omega\cdot m+n,k)$ for finite $k$, $m$ and $n$, was solved more 
than $40$ years ago. It proceeds by reducing the problem to a question about finite graphs that can be effectively, albeit unfeasibly, solved with a computer. This was announced 
without proof in \cite{HS} and \cite{haddadsabbagh}. See \cite{caicedohs} for further details.
\end{remark}

\begin{remark}\label{remark:zf}
At the cost of a somewhat more cumbersome approach, we may eliminate the use of the non-principal ultrafilter and any appeal to the axiom of choice throughout the paper. Rather 
than presenting these versions of the proofs, we mention a simple and well-known absoluteness argument ensuring that choice is indeed not needed.

We present the argument in the context of Lemma \ref{lemma:omega2+1}; the same approach removes in all proofs the need to use choice to get access to a non-principal ultrafilter. 
Work in $\mathsf{ZF}$. With $\beta$ as in the proof of Lemma \ref{lemma:omega2+1}, consider a coloring $c:[\beta]^2\to2$, and note that $L[c]$ is a model of choice, and that the 
definitions of $\beta$ and of homogeneous closed copies of $\omega\cdot 2+1$ and $3$ are absolute between the universe of sets and this inner model. Since $L[c]$ is a model of 
choice, the argument of Lemma \ref{lemma:omega2+1} gives us a homogeneous set as required, with the additional information that it belongs to $L[c]$. Similar arguments verify that 
no use of choice is needed for \emph{(1)--(5)} in Theorem \ref{theorem:pcl} or in the relevant portions of \cite{ttppfo}). 
\end{remark}

\section{The ordinal \texorpdfstring{$\omega^2$}{omega squared}}\label{section:omegasquared}

In this section we adapt the argument from the previous section to prove the following result.

\begin{theorem}\label{theorem:omegasquared}
If $k$ is a positive integer, then $\omega^\omega\to_{cl}( \omega^2, k)$.
\end{theorem}

Since $\omega^2$ is order-reinforcing, it follows that $R^{top}(\omega^2,k)=R^{cl}(\omega^2,k)\leq\omega^\omega$.

The ordinal $\omega^\omega$ appears essentially because $P^{cl}(\omega^\omega)_m=\omega^\omega$ for all finite $m$, allowing us to iterate the argument of Lemma 
\ref{lemma:omega2+1} infinitely many times.

Again we require a classical ordinal Ramsey result. This one is due to Specker \cite{specker} (see also \cite{haddadsabbagh}).

\begin{theorem}[Specker]\label{theorem:ordinaryomegasquared}
If $k$ is a positive integer, then $\omega^2\to(\omega^2,k)^2$.
\end{theorem}

\begin{proof}[Proof of Theorem \ref{theorem:omegasquared}]
The proof is by induction on $k$. The cases $k\leq 2$ are trivial. For the inductive step, suppose $k\geq 3$. Fix a non-principal ultrafilter $\mathcal U$ on $\omega$ and let $c:[\omega^
\omega]^2\to\{\text{red},\text{blue}\}$ be a coloring.

We argue in much the same way as in the proof of Lemma \ref{lemma:omega2+1}. Among the first $\omega^{k-1}+1$ elements of $\omega^\omega$, we may assume that we have a 
red-homogeneous closed copy of $\omega+1$. Let $x_0$ be its largest point and let $H_0$ be its subset of order type $\omega$. Write the set of points $\geq\omega^{k-1}+1$ as a 
disjoint union $A_1\cup A_2\cup A_3$ as in the proof of Lemma \ref{lemma:omega2+1}. If either $A_1$ or $A_2$ contains a closed copy of $\omega^\omega$, then by the inductive 
hypothesis we may assume it contains a blue-homogeneous set of $k-1$ points, and we are done by the definitions of $A_1$ and $A_2$. But $P^{cl}(\omega^\omega)_3=
\omega^\omega$, so we may assume that $A_3$ contains a closed copy $X_1$ of $\omega^\omega$.

We can now work within $X_1$ and iterate this argument infinitely many times to obtain a closed set
\[H_0\cup\{x_0\}\cup H_1\cup\{x_1\}\cup\dots\]
of order type $\omega^2$. For each $i\in\omega$, write $H_i=\{h_{i,n}\mid n\in\omega\}$ with $h_{i,0}<h_{i,1}<\dots$. By construction, for all $i,j\in\omega$ with $i<j$,
\begin{enumerate}
\item
$c(\{x_i,x_j\})=\text{red}$,
\item
$c(\{x_i,h_{j,n}\})=\text{red}$ for all $n\in\omega$ and
\item\label{item:star}
$\{n\in\omega\mid c(\{h_{i,n},x_j\})=\text{red}\}\in\mathcal U$.
\newcounter{enumTemp}
\setcounter{enumTemp}{\theenumi}
\end{enumerate}

We would like to be able to assume that condition \ref{item:star} can be strengthened to $c(\{h_{i,n},x_j\})=\text{red}$ for all $n\in\omega$ by using the ultrafilter to pass to a subset. 
However, for each $i$ there are infinitely many $j>i$, and we can only use the ultrafilter to deal with finitely many of these.

In order to overcome this difficulty, we extend the previous argument in two ways. The first new idea is to modify our construction so as to ensure that for all $i,j\in\omega$ with $i<j$, we 
also have
\begin{enumerate}
\setcounter{enumi}{\theenumTemp}
\item\label{item:starstar}
$c(\{h_{i,n},x_j\})=\text{red for all }n<j$.
\end{enumerate}
We can achieve this by modifying the construction of $X_j$ (the closed copy of $\omega^\omega$ from which we extracted $H_j\cup\{x_j\}$). Explicitly, we now include in our disjoint 
union one additional set for each pair $(i,n)$ with $i,n<j$, which contains the points $y$ that remain with $c(\{h_{i,n},y\})=\text{blue}$. We then extract $X_j$ using the fact that 
$P^{cl}(\omega^\omega)_{j^2+3}=\omega^\omega$.

This extra condition is enough for us to continue. The second new idea is to pass to a subset of the form
\[H^\prime=\{h_{i,n}\mid i\in I,n \in N\}\cup\{x_i\mid i\in I\}\]
for some infinite $I,N\subseteq\omega$, and to build up $I$ and $N$ using a back-and-forth argument. To do this, start with $I = N = \emptyset$ and add an element to $I$ and an 
element to $N$ alternately in
such a way that $c(\{h_{i,n},x_j\})=\text{red}$ whenever $n \in N$ and $i,j \in I$ with $i < j$. Condition
\ref{item:starstar} ensures that we can always add a new element to $I$ simply by taking it to be larger
than all other members of $I$ and all members of $N$ so far. Meanwhile, condition \ref{item:star} ensures there is always some $U\in\mathcal U$ from which we may choose any 
member to add to $N$: at each stage, there are
only finitely many new conditions and so our ultrafilter is enough.

Then $c(\{a,b\})=\text{red}$ whenever $a\in H^\prime$ and $b\in H^\prime\cap\{x_i\mid i\in\omega\}$. Finally, by Theorem \ref{theorem:ordinaryomegasquared} we may assume that 
there is a red-homogeneous subset $M\subseteq H^\prime\setminus\{x_i\mid i\in\omega\}$ of order type $\omega^2$, and then the topological closure of $M$ in $H^\prime$ is a 
red-homogeneous closed copy of $\omega^2$.
\end{proof}

\begin{remark} \label{remark:omegasquare}
We have organized this argument in such a way that the reader may readily verify the following. For any positive integer $k$ and any coloring $c:[\omega^\omega]^2\to\{\text{red,blue}
\}$, there is either a blue-homogeneous set of $k$ points, or a red-homogeneous closed copy of an ordinal \emph{larger} than $\omega^2$, or a red-homogeneous closed copy of 
$\omega^2$ that is moreover \emph{cofinal} in $\omega^\omega$. This strengthening of Theorem \ref{theorem:omegasquared} will be useful in Section \ref{section:omegasquared+1}.
\end{remark}

\section{The anti-tree partial ordering on ordinals}\label{section:anti-tree}

The techniques from the last few sections enable us to reach $\omega^2$, but do not seem to get us any further without cumbersome machinery. In this section we introduce a new 
approach, which provides a helpful perspective and ultimately suggests not just how to get past $\omega^2$, but even how to reach our main theorem. Here, we use this approach to 
prove the following result.

\begin{theorem}\label{theorem:topomega2}
$\omega^2\cdot 3\leq R^{top}(\omega\cdot 2,3)\leq\omega^3\cdot 100$.
\end{theorem}

It is more transparent to describe this new approach in terms of a new partial ordering on ordinals. A variant of this ordering was independently considered by Pi\~na in \cite{pina}, who 
identified countable ordinals with families of finite sets. Readers who are familiar with that work may find it helpful to note that for ordinals less than $\omega^\omega$, our new relation 
$\leq^\ast$ coincides with the superset relation $\supseteq$ under that identification. (Note that none of the results we prove here are used outside of this section.)

\begin{definition}
Let $\alpha$ and $\beta$ be ordinals. If $\beta>0$, then write $\beta=\eta+\omega^\gamma$ with $\eta$ a multiple of $\omega^\gamma$. Then we write $\alpha<^\ast\beta$ to mean 
that $\beta>0$ and $\alpha=\eta+\zeta$ for some $0<\zeta<\omega^\gamma$. We write $\alpha\lhd^\ast\beta$ to mean that $\alpha<^\ast\beta$ and there is no ordinal $\delta$ with 
$\alpha<^\ast\delta<^\ast\beta$.
\end{definition}

Equivalently, $\alpha<^\ast\beta$ if and only if $\beta=\alpha+\omega^\gamma$ for some $\gamma>\operatorname{CB}(\alpha)$, and $\alpha\lhd^\ast\beta$ if and only if $\beta=
\alpha+\omega^{\operatorname{CB}(\alpha)+1}$. For example, $\omega^3+\omega<^\ast\omega^3\cdot 2$ and $\omega^3\cdot 2+1<^\ast\omega^3\cdot 3$, but $\omega^3+\omega
\not<^\ast\omega^3\cdot 3$ and $\omega^3\cdot 2+1\not<^\ast\omega^3\cdot 2$.

Here are some simple properties of these relations.

\begin{enumerate}
\item
$<^\ast$ is a strict partial ordering.
\item
If $\alpha<^\ast\beta$ then $\alpha<\beta$ and $\operatorname{CB}(\alpha)<\operatorname{CB}(\beta)$.
\item
If $\alpha\lhd^\ast\beta$ then $\operatorname{CB}(\beta)=\operatorname{CB}(\alpha)+1$.
\item\label{item:antitreeproperty}
The class of all ordinals forms an ``anti-tree'' under the relation $<^\ast$ in the sense that for any ordinal $\alpha$, the class of ordinals $\beta$ with $\alpha<^\ast\beta$ is well-ordered 
by $<^\ast$.
\end{enumerate}

By property \ref{item:antitreeproperty}, if $k$ is a positive integer, then $\omega^k+1$ forms a tree under the relation $>^\ast$. It fact it is what we will call a \emph{perfect 
$\aleph_0$-tree of height $k$}.

\begin{definition}
Let $k$ be a positive integer and let $X$ be a single-rooted tree. We say that $x\in X$ \emph{has height $k$} to mean that $x$ has exactly $k$ predecessors. We say that $x\in X$ is a 
\emph{leaf of $X$} to mean that $x$ has no immediate successors, and denote the set of leaves of $X$ by $\ell(X)$. We say that $X$ is a \emph{perfect $\aleph_0$-tree of height $k$} 
to mean that every non-leaf of $X$ has $\aleph_0$ immediate successors and every leaf of $X$ has height $k$.

Let $X$ be a perfect $\aleph_0$-tree of height $k$. We say that a subset $Y\subseteq X$ is a \emph{full subtree of $X$} to mean that $Y$ is a perfect $\aleph_0$-tree of height $k$ 
under the induced relation.
\end{definition}

Note that if $X$ is a full subtree of $\omega^k+1$, then $X\cong\omega^k+1$. Note also that full subtrees are determined by their leaves.

\begin{figure}[!htb]
\centering
\begin{tikzpicture}[x=320,y=40]
\foreach\i in{1.2,1.44,...,6}
{
  \draw(0,0)--(1/2-1/\i,-1);
  \foreach\j in{1.2,1.44,...,6}
  {
    \draw(1/2-1/\i,-1)--({1/2-1/\i+(1/2-1/\j)/(3*\i*(\i+0.5))},-2);
  }
}
\end{tikzpicture}
\caption{A perfect $\aleph_0$-tree of height 2, corresponding to the ordinal $\omega^2+1$}
\end{figure}
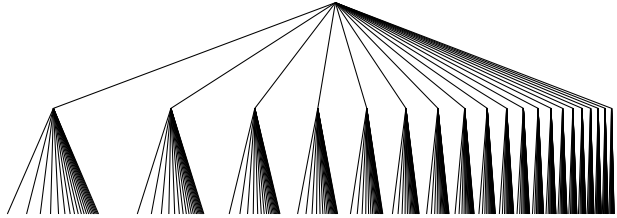

Here is a simple result about colorings of perfect $\aleph_0$-trees of height $k$. The proof essentially amounts to $k$ applications of the infinite pigeonhole principle.

\begin{lemma}\label{lemma:basicfullsubtree}
Let $k$ be a positive integer, let $X$ be a perfect $\aleph_0$-tree of height $k$ and let $c:\ell(X)\to\{\text{red},\text{blue}\}$ be a coloring. Then there exists a full subtree $Y$ of $X$ 
such that $\ell(Y)$ is monochromatic.
\end{lemma}

\begin{proof}
The proof is by induction on $k$. The case $k=1$ is simply the infinite pigeonhole principle, so assume $k>1$. Let $Z$ be the set of elements of $X$ of height at most $k-1$, so $Z$ is 
a perfect $\aleph_0$-tree of height $k-1$. Then for each $z\in\ell(Z)$, by the infinite pigeonhole principle again there exists $d(z)\in\{\text{red},\text{blue}\}$ and an infinite subset $Y_z$ 
of the successors of $z$ such that $c(x)=d(z)$ for all $x\in Y_z$. This defines a coloring $d:\ell(Z)\to\{\text{red},\text{blue}\}$, so by the inductive hypothesis there exists a full subtree 
$W$ of $Z$ such that $\ell(W)$ is monochromatic for $d$. Finally let
\[Y=W\cup\bigcup_{w\in\ell(W)}Y_w.\]
Then $Y$ is as required.
\end{proof}

Recall Theorem \ref{theorem:omega+1}, which states that if $k$ is a positive integer, then $R^{top}(\omega+1,k+1)=R^{cl}(\omega+1,k+1)\geq\omega^k+1$.
To illustrate the relevance of these notions, we now provide a second proof of this result. This is also the proof that we will mirror when the result is generalized in Theorem 
\ref{theorem:generalomegasquared+1}.

The crux of the proof is the following result, which says that any $\{\text{red},\text{blue}\}$-coloring of $\omega^k+1$ avoiding both a red-homogeneous topological copy of $\omega+1$ 
and a blue-homogeneous topological copy of $\omega$ is in some sense similar to the $k+1$-partite $\{\text{red},\text{blue}\}$-coloring that falls out of the proofs of Propositions 
\ref{proposition:lowerbound} and \ref{proposition:pomega+1}.

\begin{lemma}\label{lemma:fullsubtree}
Let $k$ be a positive integer and let $c:[\omega^k+1]^2\to\{\text{red},\text{blue}\}$ be a coloring. Suppose that
\begin{enumerate}[(a)]
\item\label{item:omegaplusonecondition}
there is no red-homogeneous topological copy of $\omega+1$, and
\item\label{item:omegacondition}
there is no blue-homogeneous topological copy of $\omega$.
\end{enumerate}
Under these assumptions, there is a full subtree $X$ of $\omega^k+1$ such that for all $x,y,z\in X$:
\begin{enumerate}
\item\label{item:antitreecondition}
if $x\lhd^\ast z$ and $y\lhd^\ast z$ then $c(\{x,y\})=\text{red}$; and
\item\label{item:antitreechildcondition}
if $x<^\ast y$ then $c(\{x,y\})=\text{blue}$.
\end{enumerate}
\end{lemma}

The proof makes use of Lemma \ref{lemma:basicfullsubtree}.

\begin{proof}
The proof is by induction on $k$.

For the base case, suppose $k=1$. By the infinite Ramsey theorem there exists an infinite homogeneous subset $Y\subseteq\omega$. By condition \emph{(\ref{item:omegacondition})} 
this must be red-homogeneous. Now by the infinite pigeonhole principle there must exist $i\in\{\text{red},\text{blue}\}$ and an infinite subset $Z\subseteq Y$ such that $c(\{x,\omega\})
=i$ for all $x\in Z$. By condition \emph{(\ref{item:omegaplusonecondition})} we must have $i=\text{blue}$. Then $Z\cup\{\omega\}$ is a full subtree of $\omega+1$ with the required 
properties.

For the inductive step, suppose $k>1$. First apply the inductive hypothesis to obtain a full subtree $Y_m$ of $\left[\omega^k\cdot m+1,\omega^k\cdot(m+1)\right]\cong\omega^{k-1}+1$ 
for each $m\in\omega$, and let $Y=\bigcup_{m\in\omega}Y_m\cup\{\omega^k\}$. Then use the inductive hypothesis again to obtain a full subtree $Z$ of $Y\setminus\ell(Y)\cong
\omega^{k-1}+1$, and let
\[W=Z\cup\{y\in\ell(Y)\mid y\lhd^\ast z\text{ for some $z\in Z$}\}.\]
By our uses of the inductive hypothesis, conditions \emph{(\ref{item:antitreecondition})} and \emph{(\ref{item:antitreechildcondition})} hold whenever $x,y,z\in Y_m$ for some $m\in
\omega$ or $x,y,z\in Z$. Thus it is sufficient to find a full subtree $X$ of $W$ such that $c(\{x,\omega^k\})=\text{blue}$ for all $x\in\ell(X)$. To do this, define a coloring $\widetilde c:
\ell(W)\to\{\text{red},\text{blue}\}$ by $\widetilde c(x)=c(\{x,\omega^k\})$. Then apply Lemma \ref{lemma:basicfullsubtree} to obtain $i\in\{\text{red},\text{blue}\}$ and a full subtree $X$ of 
$W$ such that $c(\{x,\omega^k\})=i$ for all $x\in\ell(X)$. Now let $V$ be a cofinal subset of $\ell(X)$ of order type $\omega$. By the infinite Ramsey theorem there exists an infinite 
homogeneous subset $U\subseteq V$, which by condition (\ref{item:omegacondition}) must be red-homogeneous. But then $U\cup\{\omega^k\}$ is a topological copy of $\omega+1$, 
so by condition (\ref{item:omegaplusonecondition}) we must have $i=\text{blue}$, and we are done.
\end{proof}

Theorem \ref{theorem:omega+1} now follows easily.

\begin{proof}[Second proof of Theorem \ref{theorem:omega+1}]
As in the first proof, $R^{top}(\omega+1,k+1)=R^{cl}(\omega+1,k+1)\geq\omega^k+1$.

To see that $\omega^k+1\to_{top}(\omega+1,k+1)^2$, let $c:[\omega^k+1]^2\to\{\text{red},\text{blue}\}$ be a coloring. If there is a red-homogeneous topological copy of $\omega+1$ or 
a blue-homogeneous topological copy of $\omega$, then we are done. Otherwise, choose $X\subseteq\omega^k+1$ as in Lemma \ref{lemma:fullsubtree}. Then any branch (i.e., 
maximal chain under $>^\ast$) of $X$ forms a blue-homogeneous set of $k+1$ points.
\end{proof}

We conclude this section by proving our bounds on $R^{top}(\omega\cdot 2,3)$. Before doing this, we remark that it is crucial that we consider here the topological rather than the 
closed Ramsey number. Since $\omega+n\cong\omega+1$ for every positive integer $n$, from a topological perspective, $\omega\cdot 2$ is the simplest ordinal space larger than 
$\omega+1$. Moreover, there are sets of ordinals containing a topological copy of $\omega\cdot 2$ but not even a closed copy of $\omega+2$, such as $(\omega\cdot 2+1)\setminus
\{\omega\}$. Accordingly, we have only been able to apply the technique we present here to this simplest of cases. Nonetheless, it may still be possible to adapt this technique to obtain 
upper bounds on closed (as well as topological) Ramsey numbers.

We begin by proving the lower bound. Recall from Fact \ref{fact:pcl} that $P^{top}(\omega\cdot 2)_2=\omega^2\cdot 2$. Thus we have indeed improved upon the lower bound given by 
Proposition \ref{proposition:lowerbound}. As with the lower bound of Lemma \ref{lemma:omegaplus2lower}, we provide a simple coloring based on a small finite graph.

\begin{lemma}\label{lemma:omega2lower}
$R^{top}(\omega\cdot 2,3)\geq\omega^2\cdot 3$.
\end{lemma}

\begin{proof}
Since any ordinal less than $\omega^2\cdot 3$ is homeomorphic to a subspace of $\omega^2\cdot 2+1$, it is sufficient to prove that $\omega^2\cdot 2+1\not\to_{top}(\omega\cdot 2, 
3)^2$. To see this, let $G$ be the graph represented by the following diagram.
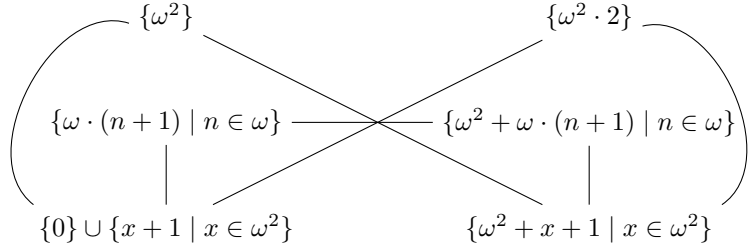
\begin{center}
\begin{tikzpicture}[x=80,y=40]
\node(bottom1)at(-1,-1){$\{0\}\cup\{x+1\mid x\in\omega^2\}$};
\node(middle1)at(-1,0){$\{\omega\cdot(n+1)\mid n\in\omega\}$};
\node(top1)at(-1,1){$\{\omega^2\}$};
\node(bottom2)at(1,-1){$\{\omega^2+x+1\mid x\in\omega^2\}$};
\node(middle2)at(1,0){$\{\omega^2+\omega\cdot(n+1)\mid n\in\omega\}$};
\node(top2)at(1,1){$\{\omega^2\cdot 2\}$};
\path(top1)edge[bend right=80](bottom1);
\path(top2)edge[bend left=80](bottom2);
\path(middle1)edge(bottom1);
\path(middle2)edge(bottom2);
\path(top1)edge(bottom2);
\path(top2)edge(bottom1);
\path(middle1)edge(middle2);
\end{tikzpicture}
\end{center}
Now define a coloring $c:[\omega^2\cdot 2+1]^2\to\{\text{red},\text{blue}\}$ by setting $c(\{x,y\})=\text{blue}$ if and only if $x$ and $y$ lie in distinct, adjacent vertices of $G$. First note 
that there is no blue triangle since $G$ is triangle-free. To see that there is no red-homogeneous topological copy of $\omega\cdot 2$, first note that every vertex of $G$ is discrete, and 
moreover the union of any vertex from the left half of $G$ and any vertex the right half of $G$ is also discrete. Therefore the only maximal red-homogeneous subspaces that are not 
discrete are $\{\omega\cdot (n+1)\mid n\in\omega\}\cup\{\omega^2\}\cup\{\omega^2\cdot 2\}$ and $\{\omega^2+\omega\cdot (n+1)\mid n\in\omega\}\cup\{\omega^2\}\cup
\{\omega^2\cdot 2\}$, neither of which contains a topological copy of $\omega\cdot 2$.
\end{proof}

Our upper bound makes use of several classical ordinal Ramsey results.

Firstly, we use the special case of Theorem \ref{theorem:ordinaryomegasquared} that $\omega^2\to(\omega^2,3)^2$.

Secondly, we use Theorem \ref{theorem:ordinarylessthanomegasquaredprecise}. In fact, we essentially prove that $R^{top}(\omega\cdot 2,3)\leq\omega^3\cdot R(K_{10}^*,L_3)$. The 
best known upper bound on $R(K_{10}^*,L_3)$ is due to Larson and Mitchell \cite{MR1630775}.
\begin{theorem}[Larson-Mitchell]
If $n>1$ is a positive integer, then $R(K_n^*,L_3)\leq n^2$.
\end{theorem}
In particular, $R(K_{10}^*,L_3)\leq 100$ and hence $\omega\cdot 100\to(\omega\cdot 10,3)^2$ by Theorem \ref{theorem:ordinarylessthanomegasquaredprecise}.

Finally, we use the following result, which was claimed without proof by Haddad and Sabbagh \cite{haddadsabbagh} and has since been proved independently by Weinert \cite{weinert}.
\begin{theorem}[Haddad-Sabbagh; Weinert]\label{theorem:ordinaryomegasquared2}
$R(\omega^2\cdot 2,3)=\omega^2\cdot 10$.
\end{theorem}

We are now ready to prove our upper bound. The first part of the proof is the following analogue of Lemma \ref{lemma:fullsubtree}. The proof uses in an essential manner that we are 
looking for a topological rather than a closed copy of $\omega\cdot 2$.

\begin{lemma}\label{lemma:fullsubtreeanalogue}
Let $c:[\omega^2+1]^2\to\{\text{red},\text{blue}\}$ be a coloring. Suppose that
\begin{enumerate}[(a)]
\item\label{item:omegatimestwocondition}
there is no red-homogeneous topological copy of $\omega\cdot 2$, and
\item\label{item:trianglecondition}
there is no blue triangle.
\end{enumerate}
Under these assumptions, there is a full subtree $X$ of $\omega^2+1$ such that for all $x,y\in X$:
\begin{enumerate}
\item\label{item:secondantitreechildcondition}
if $\operatorname{CB}(x)=\operatorname{CB}(y)$ then $c(\{x,y\})=\text{red}$;
\item\label{item:secondantitreecondition}
if $\operatorname{CB}(x)=0$ then $c(\{x,\omega^2\})=\text{blue}$; and
\item\label{item:thirdantitreecondition}
if $\operatorname{CB}(x)=1$ then $c(\{x,\omega^2\})=\text{red}$.
\end{enumerate}
\end{lemma}

\begin{proof}
First note that $(\omega^2+1)\setminus(\omega^2+1)^\prime$ has order type $\omega^2$, so by condition \emph{(\ref{item:trianglecondition})} it has a red-homogeneous subset $W$ 
of order type $\omega^2$, since $\omega^2\to(\omega^2,3)^2$. Let $Y_0$ be a full subtree of $\omega^2+1$ with $\ell(Y_0)\subseteq W$. By applying the infinite Ramsey theorem to 
$Y_0^\prime\setminus Y_0^{(2)}$, we may similarly pass to a full subtree $Y_1$ of $Y_0$ such that $Y_1^\prime\setminus Y_1^{(2)}$ is red-homogeneous. Thus $Y_1$ satisfies 
condition \emph{(\ref{item:secondantitreechildcondition})}.

Next apply Lemma \ref{lemma:basicfullsubtree} to obtain $i\in\{\text{red},\text{blue}\}$ and a full subtree $Z$ of $Y_1$ such that $c(\{x,\omega^2\})=i$ for all $x\in\ell(Z)$. If $i=
\text{red}$, then $\ell(Z)\cup\{\omega^2\}$ would be red-homogeneous, so by condition \emph{(\ref{item:omegatimestwocondition})}, $i=\text{blue}$ and $Z$ satisfies condition 
\emph{(\ref{item:secondantitreecondition})}.

Finally apply the infinite pigeonhole principle to $Z^\prime\setminus Z^{(2)}$ to obtain $j\in\{\text{red},\text{blue}\}$ and a full subtree $X$ of $Z$ such that $c(\{x,\omega^2\})=j$ for all 
$x\in X$ with $\operatorname{CB}(x)=1$. If $j=\text{blue}$, then by condition \emph{(\ref{item:trianglecondition})} we would have $c(\{x,y\})=\text{red}$ for all $x,y\in X$ with 
$\operatorname{CB}(x)=0$ and $\operatorname{CB}(y)=1$, whence $X\setminus\{\omega^2\}$ would be red-homogeneous. Hence by condition 
\emph{(\ref{item:omegatimestwocondition})}, $j=\text{red}$ and $X$ is as required.
\end{proof}

We can now complete the proof of our upper bound and hence of Theorem \ref{theorem:topomega2}.

\begin{proof}[Proof of Theorem \ref{theorem:topomega2}]
By Lemma \ref{lemma:omega2lower}, it remains only to prove that $R^{top}(\omega\cdot 2,3)\leq\omega^3\cdot 100$.

Let $X=\omega^3\cdot 100$, let $c:[X]^2\to\{\text{red},\text{blue}\}$ be a coloring and suppose by contradiction that there is no red-homogeneous topological copy of $\omega\cdot 2$ 
and no blue triangle.

First note that $X^{(2)}\setminus X^{(3)}$ has order type $\omega\cdot 100$, so it has a red-homogeneous subset $U$ of order type $\omega\cdot 10$, since $\omega\cdot 100
\to(\omega\cdot 10,3)^2$. Next let $V=\{x\in X\mid x\lhd^\ast y\text{\ for some }y\in U\}$. Note that $V$ has order type $\omega^2\cdot 10$, and so $V$ has a red-homogeneous subset 
$W$ of order type $\omega^2\cdot 2$, since $\omega^2\cdot 10\to(\omega^2\cdot 2,3)^2$ by Theorem \ref{theorem:ordinaryomegasquared2}. Finally, let
\[Y=\operatorname{cl}(W)\cup\{x\in X\mid \text{$x\lhd^\ast y$ for some $y\in W$}\},\]
where $\operatorname{cl}$ denotes the topological closure operation. Replacing $Y$ with $Y\setminus\{\operatorname{max}Y\}$ if necessary, we may then assume that $Y\cong
\omega^3\cdot 2$, and by construction both $Y^{(2)}\setminus Y^{(3)}$ and $Y^\prime\setminus Y^{(2)}$ are red-homogeneous.

Assume for notational convenience that $Y=\omega^3\cdot 2$. By applying Lemma \ref{lemma:fullsubtreeanalogue} to the interval $\left[\omega^2\cdot\alpha+1,\omega^2\cdot
(\alpha+1)\right]$ for each $\alpha\in\omega\cdot 2$, we may assume that $$c(\{\omega^2\cdot\alpha+\omega\cdot(n+1),\omega^2\cdot(\alpha+1)\})=\text{red}$$ for all $n\in\omega$. 
By applying Lemma \ref{lemma:fullsubtreeanalogue} to $(\omega^3+1)^\prime$, we may then assume that $c(\{\omega^2\cdot(\alpha+1),\omega^3\})=\text{red}$ and 
$c(\{\omega^2\cdot\alpha+\omega\cdot(n+1),\omega^3\})=\text{blue}$ for all $\alpha,n\in\omega$.

Finally by applying the infinite pigeonhole principle to $\{\omega^2\cdot(\alpha+1)\mid \alpha\in[\omega,\omega\cdot 2)\}$, we may assume that $$c(\{\omega^3,\omega^2\cdot
(\alpha+1)\})=i$$ for all $\alpha\in[\omega,\omega\cdot 2)$, where $i\in\{\text{red},\text{blue}\}$, and then by applying the infinite pigeonhole principle to $\{\omega\cdot(n+1)\mid n\in
\omega\}$, we may assume that $$c(\{\omega\cdot(n+1),\omega^3+\omega^2\})=j,$$ where $j\in\{\text{red},\text{blue}\}$. Now if $i=\text{red}$, then $(\omega^3\cdot 2)^{(2)}$ would 
be a red-homogeneous topological copy of $\omega\cdot 2$, and if $j=\text{red}$, then $$\{\omega\cdot(n+1)\mid n\in\omega\}\cup\{\omega^3+\omega\cdot(n+1)\mid n\in\omega\}\cup
\{\omega^3+\omega^2\}$$ would be a red-homogeneous topological copy of $\omega\cdot 2$. So $i=j=\text{blue}$. But then $\{\omega,\omega^3,\omega^3+\omega^2\}$ is a blue 
triangle.
\end{proof}

\begin{remark} \label{remark:omer}
Contrast this upper bound with $R^{cl}(\omega\cdot 2,3)\le\omega^4\cdot 2$ (see Proposition \ref{remark:omega8}). Very recently, Omer Mermelstein has produced a draft 
\cite{Mer17} where, using a careful topological analysis building in part on the ideas from this section, he obtains that $R^{cl}(\omega\cdot 2,3)= \omega^3\cdot 2$, in particular 
improving our upper bounds in both the topological and the closed case.
\end{remark}

\section{The ordinal \texorpdfstring{$\omega^2+1$}{omega squared plus one}}\label{section:omegasquared+1}

We now use our earlier result on $\omega^2$ together with some of the ideas from the previous section to obtain upper bounds for $\omega^2+1$. We deduce these from Theorem 
\ref{theorem:omegasquared} and the following general result.

\begin{theorem}\label{theorem:generalomegasquared+1}
Let $\alpha$ and $\beta$ be countable ordinals with $\beta>0$, let $k$ be a positive integer, and suppose they satisfy a ``cofinal version'' of
\[\omega^{\omega^\alpha}\to_{cl}(\omega^\beta,k+2)^2.\]
Then
\[\omega^{\omega^\alpha\cdot(k+1)}+1\to_{cl}(\omega^\beta+1,k+2)^2.\]
Moreover, if $\omega^{\omega^\alpha}>\omega^\beta$, then in fact
\[\omega^{\omega^\alpha\cdot k}+1\to_{cl}(\omega^\beta+1,k+2)^2.\]
The cofinal version of the partition relation requires that for every coloring $c:\left[\omega^{\omega^\alpha}\right]^2\to\{\text{red},\text{blue}\}$,
\begin{itemize}
\item
there is a blue-homogeneous set of $k+2$ points, or
\item
there is a red-homogeneous closed copy of $\omega^\beta$ \emph{that is cofinal in $\omega^{\omega^\alpha}$}, or
\item
there is already a red-homogeneous closed copy of $\omega^\beta+1$.
\end{itemize}
\end{theorem}

Before providing the proof, we first deduce our upper bounds for $\omega^2+1$. Since $\omega^2+1$ is order-reinforcing, it follows that $R^{top}(\omega^2+1,k+2)=R^{cl}
(\omega^2+1,k+2)\leq\omega^{\omega\cdot k}+1$.

\begin{corollary}\label{corollary:omegasquared+1}
If $k$ is a positive integer, then $\omega^{\omega\cdot k}+1\to_{cl}(\omega^2+1,k+2)^2$.
\end{corollary}

\begin{proof}
By Theorem \ref{theorem:generalomegasquared+1}, since $\omega^\omega>\omega^2$ it is enough to prove the cofinal version of $\omega^\omega\to_{cl}(\omega^2,k+2)^2$. The 
usual version is precisely Theorem \ref{theorem:omegasquared}, and the cofinal version is easily obtained from the same proof, as indicated in Remark \ref{remark:omegasquare}.
\end{proof}

Observe that by applying Ramsey's theorem instead of Theorem \ref{theorem:omegasquared}, one obtains yet another proof of Theorem \ref{theorem:omega+1} from the case 
$\alpha=0$. Indeed, our proof of Theorem \ref{theorem:generalomegasquared+1} is similar to our second proof of that result, though we do not explicitly use any of our results on the 
anti-tree partial ordering.

The bulk of the proof of Theorem \ref{theorem:generalomegasquared+1} is in the following result, which is our analogue of Lemma \ref{lemma:fullsubtree}. The proof makes detailed 
use of the topological structure of countable ordinals. In particular, we use two arguments due to Weiss from the proof of \cite[Theorem 2.3]{baumgartner}. It may be helpful for the 
reader to first study that proof.

\begin{lemma}\label{lemma:generalomegasquared+1}
Let $\alpha$, $\beta$ and $k$ be as in Theorem \ref{theorem:generalomegasquared+1}. Let $l$ be a positive integer and let $c:\left[\omega^{\omega^\alpha\cdot l}+1\right]^2\to
\{\text{red},\text{blue}\}$ be a coloring. Suppose that
\begin{enumerate}
\item\label{item:nohomeoassumption}
there is no red-homogeneous closed copy of $\omega^\beta+1$, and
\item\label{item:nofiniteassumption}
there is no blue-homogeneous set of $k+2$ points.
\end{enumerate}
Under these assumptions, there exists a cofinal subset $X\subseteq\omega^{\omega^\alpha\cdot l}$ such that $X$ is a closed copy of $\omega^{\omega^\alpha\cdot l}$ and $c(\{x,
\omega^{\omega^\alpha\cdot l}\})=\text{blue}$ for all $x\in X$.
\end{lemma}

\begin{proof}
The proof is by induction on $l$.

For the case $l=1$, since $\omega^{\omega^\alpha}\to_{cl}(\omega^{\omega^\alpha})^1_2$, there exists $X\subseteq\omega^{\omega^\alpha}$ and $i\in\{\text{red},\text{blue}\}$ such 
that $X$ is a closed copy of $\omega^{\omega^\alpha}$ (and therefore $X$ is cofinal in $\omega^{\omega^\alpha}$) and $c(\{x,\omega^{\omega^\alpha}\})=i$ for all $x\in X$. Suppose 
for contradiction that $i=\text{red}$. By our assumptions together with the definition of the cofinal version of the partition relation, there exists a cofinal subset $Y\subseteq X$ such that 
$Y$ is a closed copy of $\omega^\beta$ and $[Y]^2\subseteq c^{-1}(\{\text{red}\})$. But then $Y\cup\{\omega^{\omega^\alpha}\}$ is a red-homogeneous closed copy of $\omega^
\beta+1$, contrary to assumption \ref{item:nohomeoassumption}. Hence $i=\text{blue}$ and we are done.

For the inductive step, suppose $l>1$. Let
\[Z=\left\{\omega^{\omega^\alpha}\cdot\gamma\mid\gamma\in\omega^{\omega^\alpha\cdot(l-1)}\setminus\{0\}\right\},\]
so $Z$ is a closed copy of $\omega^{\omega^\alpha\cdot(l-1)}$. By the inductive hypothesis, there exists a cofinal subset $Y\subseteq Z$ such that $Y$ is a closed copy of 
$\omega^{\omega^\alpha\cdot(l-1)}$ and $c(\{x,\omega^{\omega^\alpha\cdot l}\})=\text{blue}$ for all $x\in Y$. Write $Y=\{y_\delta\mid\delta\in\omega^{\omega^\alpha\cdot(l-1)}\}$ in 
increasing order. Then by Weiss's lemma \cite[Lemma 2.6]{baumgartner}, for each $\delta\in\omega^{\omega^\alpha\cdot(l-1)}$ there exists a cofinal subset $Z_\delta\subseteq
(y_\delta,y_{\delta+1})$ such that $Z_\delta$ is a closed copy of $\omega^{\omega^\alpha}$.

Now since $\omega^{\omega^\alpha}\to_{cl}(\omega^{\omega^\alpha})^1_2$, for each $\delta\in\omega^{\omega^\alpha\cdot(l-1)}$ there exists $X_\delta\subseteq Z_\delta$ and 
$i_\delta\in\{\text{red},\text{blue}\}$ such that $X_\delta$ is a closed copy of $\omega^{\omega^\alpha}$ (and therefore $X_\delta$ is cofinal in $Z_\delta$) and $c(\{x,\omega^
{\omega^\alpha\cdot l}\})=i_\delta$ for all $x\in X_\delta$. Recall now that $\omega^{\omega^\alpha\cdot(l-1)}\to(\omega^{\omega^\alpha\cdot(l-1)})^1_2$ since $\omega^{\omega^\alpha
\cdot(l-1)}$ is a power of $\omega$. It follows that there exists $S\subseteq\omega^{\omega^\alpha\cdot(l-1)}$ of order type $\omega^{\omega^\alpha\cdot(l-1)}$ and $i\in\{\text{red},
\text{blue}\}$ such that $i_\delta=i$ for all $\delta\in S$.

Suppose for contradiction that $i=\text{red}$. We now use an argument from the proof of \cite[Theorem 2.3]{baumgartner}. Let $(\delta_m)_{m\in\omega}$ be a strictly increasing cofinal 
sequence from $S$, and let $(\eta_m)_{m\in\omega}$ be a strictly increasing cofinal sequence from $\omega^\alpha$ (or let $\eta_m=0$ for all $m\in\omega$ if $\alpha=0$). For each 
$m\in\omega$, pick $W_m\subseteq X_{\delta_m}$ such that $W_m$ is a closed copy of $\omega^{\eta_m}+1$, and let $W=\bigcup_{m\in\omega}W_m$. Note that $W$ is cofinal in 
$\omega^{\omega^\alpha\cdot l}$, $W$ is a closed copy of $\omega^{\omega^\alpha}$ and $c(\{x,\omega^{\omega^\alpha\cdot l}\})=\text{red}$ for all $x\in W$. By our assumptions 
together with the definition of the cofinal version of the partition relation, there exists a cofinal subset $V\subseteq W$ such that $V$ is a closed copy of $\omega^\beta$ and 
$[V]^2\subseteq c^{-1}(\{\text{red}\})$. But then $V\cup\{\omega^{\omega^\alpha\cdot l}\}$ is a red-homogeneous closed copy of $\omega^\beta+1$, contrary to assumption 
\ref{item:nohomeoassumption}.

Therefore $i=\text{blue}$. Finally, let
\[X=\bigcup_{\delta\in S}X_\delta\cup\operatorname{cl}(\{y_{\delta+1}\mid\delta\in S\}),\]
where $\operatorname{cl}$ denotes the topological closure operation. Then the set $X$ is as required.
\end{proof}

We may now deduce Theorem \ref{theorem:generalomegasquared+1} in the much same way that we deduced Theorem \ref{theorem:omega+1} from Lemma \ref{lemma:fullsubtree}.

\begin{proof}[Proof of Theorem \ref{theorem:generalomegasquared+1}]
First assume that $\omega^{\omega^\alpha}>\omega^\beta$. We prove by induction on $l$ that for all $l\in\{1,2,\dots,k\}$,
\[\omega^{\omega^\alpha\cdot l}+1\to_{cl}(\omega^\beta+1,l+2)^2.\]

In every case, if either of the assumptions in Lemma \ref{lemma:generalomegasquared+1} does not hold, then we are done since $l\leq k$. We may therefore choose $X$ as in Lemma 
\ref{lemma:generalomegasquared+1}.

For the base case $l=1$, to avoid a blue triangle, $X$ must be red-homogeneous. But $X$ contains a closed copy of $\omega^\beta+1$ since $\omega^{\omega^\alpha}\geq\omega^
\beta+1$, and so we are done.

For the inductive step, suppose $l\geq 2$. Then $X$ has a closed copy $Y$ of $\omega^{\omega^\alpha\cdot(l-1)}+1$. By the inductive hypothesis, either $Y$ contains a 
red-homogeneous closed copy of $\omega^\beta+1$, in which case we are done, or $Y$ contains a blue-homogeneous set $Z$ of $l+1$ points. But in that case $Z\cup
\{\omega^{\omega^\alpha\cdot l}\}$ is a blue-homogeneous set of $l+2$ points, and we are done.

Finally, if we cannot assume that $\omega^{\omega^\alpha}>\omega^\beta$, then the base case breaks down. However, we may instead use the base case $\omega^{\omega^\alpha}
+1\to_{cl}(\omega^\beta+1,2)^2$, which follows from the fact that $\omega^{\omega^\alpha}\geq\omega^\beta$. The inductive step is then identical.
\end{proof}

\section{The weak topological Erd\H{o}s-Milner theorem}\label{section:wtem}

Finally we reach our main result, which demonstrates that $R^{top}(\alpha,k)$ and $R^{cl}(\alpha,k)$ are countable for all countable $\alpha$ and all finite $k$.

This is a topological version of a classical result due to Erd\H{o}s and Milner \cite{erdosmilner}. Before stating it, we first provide a simplified proof of the classical version.

\begin{theorem}[Weak Erd\H{o}s-Milner]
Let $\alpha$ and $\beta$ be countable nonzero ordinals, and let $k>1$ be a positive integer. If
\[\omega^\alpha\to(\omega^{1+\beta},k)^2,\]
then
\[\omega^{\alpha+\beta}\to(\omega^{1+\beta},k+1)^2.\]
\end{theorem}

Since trivially $\omega^{1+\alpha}\to(\omega^{1+\alpha},2)^2$, it follows by induction on $k$ that $R(\omega^{1+\alpha},k+1)\leq\omega^{1+\alpha\cdot k}$ for all countable $\alpha$ 
and finite $k$. In fact, Erd\H{o}s and Milner proved a stronger version of the above theorem, in which $k+1$ is replaced by $2k$, implying that $R(\omega^{1+\alpha},2^k)\leq\omega^
{1+\alpha\cdot k}$. This is why we use the adjective \emph{weak} here.

Our proof is essentially a simplified version of the original, which can be found in \cite[Theorem 7.2.10]{W}. The basic idea is to write $\omega^{\alpha+\beta}$ as a sequence of 
$\omega^\beta$ intervals, each of order type $\omega^\alpha$, to list these intervals in type $\omega$, and to recursively build up a red-homogeneous copy of $\omega^\beta$ 
consisting of one element from each interval. This would achieve the above theorem with $1+\beta$ weakened to $\beta$. To obtain a copy of $\omega^{1+\beta}$, we simply choose 
infinitely many elements from each interval instead of just one. In the proof we use the fact that $\omega^\alpha\to(\omega^\alpha)^1_m$ for all finite $m$.

\begin{proof}
Let $c:\left[\omega^{\alpha+\beta}\right]^2\to\{\text{red},\text{blue}\}$ be a coloring.

For each $x\in\omega^\beta$, let
\[I_x=\left[\omega^\alpha\cdot x,\omega^\alpha\cdot(x+1)\right),\]
so $I_x$ has order type $\omega^\alpha$. Let $(x_n)_{n\in\omega}$ be a sequence of points from $\omega^\beta$ in which every member of $\omega^\beta$ appears infinitely many 
times. We attempt to inductively build a red-homogeneous set $A=\{a_n\mid n\in\omega\}$ of order-type $\omega^{1+\beta}$ with $a_n\in I_{x_n}$ for every $n\in\omega$.

Suppose that we have chosen $a_1,a_2,\dots,a_{m-1}$ for some $m\in\omega$. Let
\[P=\{a_n\mid n\in\{1,2,\dots,m-1\},a_n\in I_{x_m}\},\]
and let
\[
J=
\begin{cases}
I_{x_m},&\text{if $P=\emptyset$}\\
I_{x_m}\setminus[0,\max P],&\text{if $P\neq\emptyset$},
\end{cases}
\]
so $J$ has order type $\omega^\alpha$. For each $n\in\{1,2,\dots,m-1\}$ let
\[J_n=\{a\in J\mid c(\{a_n,a\})=\text{blue}\}.\]
If $\bigcup_{n=1}^{m-1}J_n=J$, then since $\omega^\alpha\to(\omega^\alpha)^1_{m-1}$, $J_n$ has order type $\omega^\alpha$ for some $n\in\{1,2,\dots,m-1\}$. Then since $\omega^
\alpha\to(\omega^{1+\beta},k)^2$, $J_n$ either has a red-homogeneous subset of order type $\omega^{1+\beta}$, in which case we are done, or a blue-homogeneous subset $B$ of 
$k$ points, in which case $B\cup\{a_n\}$ is a blue-homogeneous set of $k+1$ points, and we are done. Thus we may assume that $J\setminus\bigcup_{n=1}^{m-1}J_n\neq\emptyset$, 
and choose $a_m$ to be any member of this set.

Clearly the resulting set $A$ is red-homogeneous. To see that it has order type $\omega^{1+\beta}$, simply observe that by choice of $(x_n)_{n\in\omega}$, $A\cap I_x$ is infinite for all 
$x\in\omega^\beta$, and that by choice of $J$, $A\cap I_x$ in fact has order type $\omega$ for all $x\in\omega^\beta$.
\end{proof}

Here is our topological version of the weak Erd\H{o}s-Milner theorem.

\begin{theorem}[The weak topological Erd\H{o}s-Milner theorem]\label{theorem:wtem}
Let $\alpha$ and $\beta$ be countable nonzero ordinals, and let $k>1$ be a positive integer. If
\[\omega^{\omega^\alpha}\to_{top}(\omega^\beta,k)^2,\]
then
\[\omega^{\omega^\alpha\cdot\beta}\to_{top}(\omega^\beta,k+1)^2.\]
\end{theorem}

In particular, as we shall deduce in Corollary \ref{corollary:wtem}, $R^{top}(\alpha,k)$ and $R^{cl}(\alpha,k)$ are countable whenever $\alpha$ is a countable ordinal and $k$ is a 
positive integer.

Our proof follows the same outline as our proof of the classical version, except that we use intervals in the sense of the anti-tree partial ordering rather than in the usual sense. 
Furthermore, rather than constructing a closed copy of $\omega^\beta$ directly, we instead construct a larger set and then thin it out. As in the previous section, we make detailed use 
of the structure of countable ordinals, including an argument from the proof of \cite[Lemma 2.6]{baumgartner}. Note that the proof does not directly use any of our previous results.

\begin{proof}
Let $c:\left[\omega^{\omega^\alpha\cdot\beta}\right]^2\to\{\text{red},\text{blue}\}$ be a coloring.

First of all, fix a strictly increasing cofinal sequence $(\gamma_n)_{n\in\omega}$ from $\omega^{\omega^\alpha}$.

Define an indexing set of pairs
\[S=\left\{(x,y)\mid x\in\beta,\omega^{\omega^\alpha\cdot(x+1)}\cdot y\in\omega^{\omega^\alpha\cdot\beta}\right\},\]
and for each $(x,y)\in S$, let
\[X_{(x,y)}=\left\{\omega^{\omega^\alpha\cdot(x+1)}\cdot y+\omega^{\omega^\alpha\cdot x}\cdot z\mid z\in\omega^{\omega^\alpha}\setminus\{0\}\right\},\]
so $X_{(x,y)}$ is a closed copy of $\omega^{\omega^\alpha}$. Let $(x_n,y_n)_{n\in\omega}$ be a sequence of pairs from $S$ in which every member of $S$ appears infinitely many 
times. We attempt to inductively build a red-homogeneous set $A=\{a_n\mid n\in\omega\}$, which will contain a closed copy of $\omega^\beta$, with $a_n\in X_{(x_n,y_n)}$ for every 
$n\in\omega$.

Suppose that we have chosen $a_1,a_2,\dots,a_{m-1}$ for some $m\in\omega$. Let
\[P=\left\{a_n\mid n\in\{1,2,\dots,m-1\},a_n\in X_{(x_m,y_m)}\right\},\]
let
\[Q=P\cup\left\{\omega^{\omega^\alpha\cdot(x_m+1)}\cdot y_m+\omega^{\omega^\alpha\cdot x_m}\cdot\gamma_{|P|}\right\},\]
and let
\[Y=X_{(x_m,y_m)}\setminus[0,\max Q],\]
so $Y$ is a closed copy of $\omega^{\omega^\alpha}$. For each $n\in\{1,2,\dots,m-1\}$ let
\[Y_n=\{a\in Y\mid c(\{a_n,a\})=1\}.\]
If $\bigcup_{n=1}^{m-1}Y_n=Y$, then since $\omega^{\omega^\alpha}\to_{cl}(\omega^{\omega^\alpha})^1_{m-1}$, $Y_n$ contains a closed copy of $\omega^{\omega^\alpha}$ for 
some $n\in\{1,2,\dots,m-1\}$. Then since $\omega^{\omega^\alpha}\to_{cl}(\omega^\beta,k)^2$, $Y_n$ either contains a red-homogeneous closed copy of $\omega^\beta$, in which 
case we are done, or a blue-homogeneous set $B$ of $k$ points, in which case $B\cup\{a_n\}$ is a blue-homogeneous set of $k+1$ points, and we are done. Thus we may assume 
that $Y\setminus\bigcup_{n=1}^{m-1}Y_n\neq\emptyset$, and choose $a_m$ to be any member of this set.

Clearly the resulting set $A$ is red-homogeneous. To complete the proof, observe that by choice of $(x_n,y_n)_{n\in\omega}$, $A\cap X_{(x,y)}$ is infinite for all $(x,y)\in S$, and that 
by choice of $Q$, $A\cap X_{(x,y)}$ is a cofinal subset of $X_{(x,y)}$ of order type $\omega$ for all $(x,y)\in S$. We claim that this property is enough to ensure that $A$ contains a 
closed copy of $\omega^\beta$.

To prove the claim, for each $\delta\in[1,\beta]$ and each ordinal $y$ with $\omega^{\omega^\alpha\cdot\delta}\cdot y\in\omega^{\omega^\alpha\cdot\beta}$ we find a cofinal subset 
$C_{\delta,y}\subseteq\left[\omega^{\omega^\alpha\cdot\delta}\cdot y+1,\omega^{\omega^\alpha\cdot\delta}\cdot(y+1)\right)$ such that $C_{\delta,y}\subseteq A$ and $C_{\delta,y}$ is 
a closed copy of $\omega^\delta$; then $C_{\beta,0}\subseteq A$ is a closed copy of $\omega^\beta$, as required. We do this by induction on $\delta$.

First suppose $\delta\in[1,\beta]$ is a successor ordinal, say $\delta=x+1$. Fix an ordinal $y$ with $\omega^{\omega^\alpha\cdot\delta}\cdot y\in\omega^{\omega^\alpha\cdot\beta}$, 
and observe that $(x,y)\in S$ and that $X_{(x,y)}$ is a cofinal subset of the interval 
 $$ \left[\omega^{\omega^\alpha\cdot\delta}\cdot y+1,\omega^{\omega^\alpha\cdot\delta}\cdot (y+1)\right). $$
Recall now that $A\cap X_{(x,y)}$ is a cofinal subset of $X_{(x,y)}$ of order type $\omega$. Thus if $\delta=1$ then we may simply take $C_{1,y}=A\cap X_{(0,y)}$, so assume 
$\delta>1$, and write $A\cap X_{(x,y)}=\{b_n\mid n\in\omega\}$ in increasing order. For each $n\in\omega\setminus\{0\}$ we find a cofinal subset $D_n\subseteq(b_{n-1},b_n)$ such 
that $D_n\subseteq A$ and $D_n$ is a closed copy of $\omega^x$; then we may take $C_{\delta,y}=\bigcup_{n\in\omega\setminus\{0\}}D_n\cup\{b_n\}$. We do this using an argument 
essentially taken from the proof of \cite[Lemma 2.6]{baumgartner}. Fix $n\in\omega\setminus\{0\}$ and write $b_n=\omega^{\omega^\alpha\cdot(x+1)}\cdot y+\omega^{\omega^\alpha
\cdot x}\cdot z$ with $z\in\omega^{\omega^\alpha}\setminus\{0\}$. Let $v=\omega^{\omega^\alpha}\cdot y+z$ so that $b_n=\omega^{\omega^\alpha\cdot x}\cdot v$. If $v$ is a 
successor ordinal, say $v=u+1$, then by the inductive hypothesis we may take $D_n=C_{x,u}$. If $v$ is a limit ordinal, then let $(u_m)_{m\in\omega}$ be a strictly increasing cofinal 
sequence from $v$ with $\omega^{\omega^\alpha\cdot x}\cdot u_0\geq b_{n-1}$ and let $(\eta_m)_{m\in\omega}$ be a strictly increasing cofinal sequence from $\omega^x$. By the 
inductive hypothesis, for each $m\in\omega$ we may choose a subset $E_m\subseteq C_{x,u_m}$ such that $E_m$ is a closed copy of $\eta_m+1$. Then take $D_n=\bigcup_
{m\in\omega}E_m$.

Suppose instead $\delta\in[1,\beta]$ is a limit ordinal. Fix an ordinal $y$ with $\omega^{\omega^\alpha\cdot\delta}\cdot y\in\omega^{\omega^\alpha\cdot\beta}$. Let $(x_n)_
{n\in\omega}$ be a strictly increasing cofinal sequence from $\delta$. For each $n\in\omega$, let $\zeta_n$ be the ordinal such that $\delta=x_n+1+\zeta_n$ and let 
$y_n=\omega^{\omega^\alpha\cdot\zeta_n}\cdot y+1$. Then $\omega^{\omega^\alpha\cdot(x_n+1)}\cdot y_n=\omega^{\omega^\alpha\cdot\delta}\cdot y+\omega^{\omega^\alpha
\cdot(x_n+1)}$ and so $C_{x_n+1,y_n}\subseteq\left[\omega^{\omega^\alpha\cdot\delta}\cdot y+\omega^{\omega^\alpha\cdot(x_n+1)}+1,\omega^{\omega^\alpha\cdot\delta}\cdot 
y+\omega^{\omega^\alpha\cdot(x_n+1)}\cdot 2\right)$. By the inductive hypothesis, for each $n\in\omega$ we may choose a subset $D_n\subseteq C_{x_n+1,y_n}$ such that $D_n$
is a closed copy of $\omega^{x_n}+1$. Then take $C_{\delta,y}=\bigcup_{n\in\omega}D_n$.
\end{proof}

\begin{remark}\label{remark:strong}
We expect that as in the original Erd\H{o}s-Milner theorem, it should be possible to improve $k+1$ to $2k$ in Theorem \ref{theorem:wtem}. Indeed, the basic argument from 
\cite[Theorem 7.2.10]{W} works in the topological setting. The key technical difficulty appears to be in formulating and proving an appropriate topological version of statement (1) in that 
account.
\end{remark}

The weak topological Erd\H{o}s-Milner theorem allows us to obtain upper bounds for countable indecomposable ordinals. Before describing some of these, we first observe that by very 
slightly adapting our argument, we may obtain improved bounds for other ordinals. Here is a version for ordinals of the form $\omega^\beta\cdot m+1$.

\begin{theorem}\label{theorem:wtemsuccessor}
Let $\alpha$ and $\beta$ be countable nonzero ordinals, and let $k>1$ be a positive integer. If
\[\omega^{\omega^\alpha}\to_{cl}(\omega^\beta\cdot m+1,k)^2,\]
then
\[\omega^{\omega^\alpha\cdot\beta}\cdot R(m,k+1)+1\to_{top}(\omega^\beta\cdot m+1,k+1)^2.\]
\end{theorem}

\begin{proof}
Write $\omega^{\omega^\alpha\cdot\beta}\cdot R(m,k+1)+1$ as a disjoint union $M\cup N$, where
\[N=\left\{\omega^{\omega^\alpha\cdot\beta}\cdot(y+1)\mid y\in\{0,1,\dots,R(m,k+1)-1\}\right\},\]
and let $c:[M\cup N]^2\to\{\text{red},\text{blue}\}$ be a coloring.

First of all, we may assume that $N$ contains a red-homogeneous set of $m$ points, say $a_0,a_1,\dots,$ $a_{m-1}$. Now continue as in the proof of Theorem \ref{theorem:wtem} and 
attempt to build a red-homogeneous set $A=\{a_n\mid n\in\omega\}$, only start by including $a_0,a_1,\dots,a_{m-1}$, and then work entirely within $M$.

If we succeed, then the same proof as the one in Theorem \ref{theorem:wtem} shows that for each $y\in\{0,1,\dots,R(m,k+1)-1\}$, $A\setminus\{a_0,a_1,\dots,a_{m-1}\}$ contains a 
closed copy $C_y$ of $\omega^\beta$ that is a cofinal subset of the interval $\left[\omega^{\omega^\alpha\cdot\beta}\cdot y+1,\omega^{\omega^\alpha\cdot\beta}\cdot(y+1)\right]$. 
Writing $a_i=\omega^{\omega^\alpha\cdot\beta}\cdot(y_i+1)$ for each $i\in\{0,1,\dots,m-1\}$, we see that
\[\bigcup_{i=0}^{m-1}C_{y_i}\cup\{a_i\}\]
is a red-homogeneous closed copy of $\omega^\beta\cdot m+1$, as required.
\end{proof}

We conclude this section with some explicit upper bounds implied by our results. It is easy to verify similar results for ordinals of other forms.

\begin{corollary}\label{corollary:wtem}
Let $\alpha$ be a countable nonzero ordinal and let $k$, $m$ and $n$ be positive integers.
\begin{enumerate}
\item
$R^{top}(\omega^{\omega^\alpha},k+1)=R^{cl}(\omega^{\omega^\alpha},k+1)\leq\omega^{\omega^{\alpha\cdot k}}$.
\item
$R^{top}(\omega^{\omega^\alpha}+1,k+1)=R^{cl}(\omega^{\omega^\alpha}+1,k+1)\leq
\begin{cases}
\omega^{\omega^{\alpha\cdot k}}+1,&\text{if $\alpha$ is infinite}\\
\omega^{\omega^{(n+1)\cdot k-1}}+1,&\text{if $\alpha=n$ is finite.}
\end{cases}$
\item
$R^{top}(\omega^n\cdot m+1,k+2)=R^{cl}(\omega^n\cdot m+1,k+2)\leq\omega^{\omega^k\cdot n}\cdot R(m,k+2)+1$.
\end{enumerate}
\end{corollary}

\begin{proof}
First note that all three equalities are immediate since all ordinals considered are order-reinforcing. It remains to prove the inequalities.
\begin{enumerate}
\item\label{item:wteminduction}
This follows immediately from Theorem \ref{theorem:wtem} by induction on $k$, the case $k=1$ being trivial.
\item\label{item:wtemsuccessorinduction}
It follows from part \ref{item:wteminduction} that $\omega^{\omega^{(\alpha+1)\cdot(k-1)}}\to_{cl}(\omega^{\omega^{\alpha+1}},k)^2$ (the case $k=1$ again being trivial) and hence 
$\omega^{\omega^{(\alpha+1)\cdot(k-1)}}\to_{cl}(\omega^{\omega^\alpha}+1,k)^2$. Therefore $\omega^{\omega^{(\alpha+1)\cdot(k-1)+\alpha}}+1\to_{cl}(\omega^{\omega^\alpha}+1,
k+1)^2$ by Theorem \ref{theorem:wtemsuccessor}. If $\alpha$ is infinite, then $(\alpha+1)\cdot(k-1)+\alpha=\alpha\cdot k$, and if $\alpha=n$ is finite, then $(\alpha+1)\cdot(k-1)+
\alpha=(n+1)\cdot k-1$, as required.
\item
It follows from part \ref{item:wteminduction} that $\omega^{\omega^k}\to_{cl}(\omega^\omega,k+1)^2$ and hence $\omega^{\omega^k}\to_{cl}(\omega^n\cdot m+1,k+1)^2$. The result 
then follows from Theorem \ref{theorem:wtemsuccessor}.\qedhere
\end{enumerate}
\end{proof}

\begin{remark}
For the case in which $\alpha=n$ is a positive integer, if a cofinal version of part \ref{item:wteminduction} holds, then we could use Theorem \ref{theorem:generalomegasquared+1} to 
improve part \ref{item:wtemsuccessorinduction} to $R^{cl}(\omega^{\omega^n}+1,k+2)\leq\omega^{\omega^{n\cdot(k+1)}\cdot k}+1$.
\end{remark}

\section{Questions}\label{section:questions}

We close with a few questions. Firstly, there is typically a large gap between our lower and upper bounds, leaving plenty of room for improvement. In particular, our general lower bound 
in Proposition \ref{proposition:lowerbound} is very simple and yet is still our best bound with the exception of a couple of special cases.

Some further exact equalities could be informative. One of the key reasons why the classical results detailed in \cite{HS,haddadsabbagh,HS2,caicedohs,Mil} are more precise than our 
topological results in Sections \ref{section:steppingup}--\ref{section:omegasquared+1} is that for various $\alpha<\omega^\omega$, the computation of $R(\alpha,k)$ reduces to a
problem in finite combinatorics. Some hint of a topological version of this appears in the proof of the lower bounds of Lemmas \ref{lemma:omegaplus2lower} and 
\ref{lemma:omega2lower}, though it is far from clear how to obtain an exact equality.

\begin{question}
Is it possible to reduce the computation of $R^{top}(\alpha,k)$ or $R^{cl}(\alpha,k)$ to finite combinatorial problems, even for $\alpha<\omega^2$?
\end{question}

A partition ordinal is an ordinal $\alpha$ satisfying $\alpha\to(\alpha,3)^2$. As we have seen, $\omega$ and $\omega^2$ are partition ordinals. Other than these, every countable 
partition ordinal has the form $\omega^{\omega^\beta}$, and in the other direction, $\omega^{\omega^\beta}$ is a partition ordinal if $\beta$ has the form $\omega^\gamma$ or 
$\omega^\gamma+\omega^\delta$ \cite{schipperusord}. 

\begin{question}
Are there any countable ``topological partition'' ordinals, satisfying $\alpha\to_{top}(\alpha,3)^2$, other 
than $\omega$?
\end{question}
By our lower bound (Proposition \ref{proposition:lowerbound}), these must all have the form $\omega^{\omega^\beta}$. Since every power of $\omega$ is order-reinforcing, this 
question is equivalent to the version resulting from replacing the topological partition relation by the closed partition relation.

We expect that a strong version of the topological Erd\H{o}s-Milner theorem should hold, that is, we expect that in Theorem \ref{theorem:wtem} it should be possible to improve $k+1$ 
to $2k$. 
\begin{question}
Let $\alpha$ and $\beta$ be countable nonzero ordinals, and let $k>1$ be a positive integer. Is it the case that if 
\[\omega^{\omega^\alpha}\to_{top}(\omega^{1+\beta},k)^2,\]
then
\[\omega^{\omega^\alpha\cdot\beta}\to_{top}(\omega^{1+\beta},2k)^2?\]
\end{question}
See Remark \ref{remark:strong} for further details.

Finally, it would be nice to determine whether the existing conjecture $\omega_1\to(\alpha,k)^3$ holds, and also to discover something about the topological or closed version of this 
relation. 
\begin{question}
Is it the case that for all countable ordinals $\alpha$ and all finite $k$, $\omega_1\to_{top}(\alpha,k)^3$?
\end{question}
Again, since every power of $\omega$ is order-reinforcing, this question is equivalent if we replace the topological partition relation by the closed partition relation. If the 
answer to this question is yes, then the classical and closed relations are in fact equivalent whenever the ordinal on the left-hand side is $\omega_1$ and the ordinals on the 
right-hand side are countable.

\bibliographystyle{alpha}
\bibliography{references}

\end{document}